\documentclass[12pt]{amsart}
\usepackage{a4wide,enumerate,xcolor}
\usepackage{amsmath,graphicx,comment,enumerate}
\usepackage{mathtools} 
\allowdisplaybreaks

\usepackage{enumitem}
\setlist[itemize]{label={$\bullet$}, leftmargin=36pt, itemsep=3pt}

\let\pa\partial
\let\na\nabla
\let\eps\varepsilon

\newcommand{\R}{{\mathbb R}}
\newcommand{\diver}{\operatorname{div}}

\newtheorem{theorem}{Theorem}
\newtheorem{lemma}[theorem]{Lemma}

\newtheorem{definition}{Definition}

\begin{document}

\title[Weak--strong uniqueness for cross-diffusion systems]{
Improved weak--strong uniqueness \\
for general cross-diffusion systems}

\author[N. Geltner]{Noah Geltner}
\address{Institute of Analysis and Scientific Computing, TU Wien, Wiedner Hauptstra\ss e 8--10, 1040 Wien, Austria}
\email{noah.geltner@tuwien.ac.at} 

\author[M. Heitzinger]{Maria Heitzinger}
\address{Institute of Analysis and Scientific Computing, TU Wien, Wiedner Hauptstra\ss e 8--10, 1040 Wien, Austria}
\email{maria.heitzinger@tuwien.ac.at} 

\author[A. J\"ungel]{Ansgar J\"ungel}
\address{Institute of Analysis and Scientific Computing, TU Wien, Wiedner Hauptstra\ss e 8--10, 1040 Wien, Austria}
\email{juengel@tuwien.ac.at} 

\date{\today}

\thanks{The authors acknowledge partial support from the Austrian Science Fund (FWF), grant 10.55776/PAT2687825, and from the Austrian Federal Ministry for Women, Science and Research and implemented by \"OAD, project MULT09/2025. This work has received funding from the European Research Council (ERC) under the European Union's Horizon 2020 research and innovation programme, ERC Advanced Grant NEUROMORPH, no.~101018153. For open-access purposes, the authors have applied a CC BY public copyright license to any author-accepted manuscript version arising from this submission.} 

\begin{abstract}
The weak--strong uniqueness of bounded solutions is established for a broad class of cross-diffusion systems, with or without volume-filling effects, and with or without full coercivity. Compared to existing results, the regularity and positivity assumptions on the strong solution are significantly relaxed. The analysis also extends to non-coercive systems, with the volume-filling constraint compensating for the lack of coercivity. The proof is based on two key ideas: The positivity condition is eliminated through the construction of a {\em glued entropy}, while the regularity requirement can be relaxed by using the gradient estimate and the Gagliardo--Nirenberg inequality.
\end{abstract}

 

\maketitle


\section{Introduction}

Cross-diffusion systems arise in numerous applications, including population dynamics, semiconductor devices, and multiphase flow, where the evolution of multiple interacting species is governed by diffusion mechanisms that depend on the concentrations of all components. The strong coupling and the lack of positive definiteness of the diffusion matrix makes the analysis very challenging. Among the fundamental questions is the uniqueness of weak solutions, which is typically out of reach for nonlinear cross-diffusion systems. A powerful alternative is provided by the weak--strong uniqueness principle, asserting that a weak solution coincides with a strong solution emanating from the same initial data as long as the latter exists. Such results are commonly established by means of relative entropy methods, which exploit the underlying entropy structure of the system to quantify the distance between two solutions.

In this work, we establish improved weak--strong uniqueness results for a broad class of cross-diffusion systems possessing an entropy structure and admitting bounded weak solutions. Compared to previous contributions, we substantially weaken the regularity assumptions on the strong solution and extend the class of admissible mobility matrices by allowing the positive definiteness condition to be replaced by weaker hypocoercive assumptions.

\subsection{Model equations and definitions}

The dynamics of the densities or volume fractions $u_1,\ldots,u_N$ of a multicomponent system are given by the cross-diffusion equations
\begin{align}\label{1.eq}
  \pa_t u_i = \diver\bigg(\sum_{j=1}^N A_{ij}(u)\na u_j\bigg) + r_i(u)
  \quad\mbox{in }\Omega,\ t>0, 
\end{align}
with the initial and no-flux boundary conditions
\begin{align}\label{1.bic}
  u_i(0)=u_i^0\mbox{ in }\Omega, \quad
  \sum_{j=1}^N A_{ij}(u)\na u_j\cdot\nu=0\mbox{ on }\pa\Omega,\ t>0,
  \ i=1,\ldots,N,
\end{align}
where $\Omega\subset\R^d$ ($d\ge 1$) is a bounded domain, $\nu$ is the exterior unit normal vector of $\pa\Omega$, and $u=(u_1,\ldots,u_N)$ is the solution vector with nonnegative components. Applications include fluid mixtures in which $u_1,\ldots,u_n$ denote the volume fractions of the mixture constituents, while $u_{n+1},\ldots,u_N$ represent the concentrations of chemical species or other quantities influencing the mixture. 

The aim of this work is to prove the weak--strong uniqueness property for \eqref{1.eq}--\eqref{1.bic}. For this, we assume that the cross-diffusion system has a particular entropy structure and that the diffusion matrix is non-degenerate. More precisely, we suppose that the entropy density associated to the cross-diffusion system is assumed to be given by 
\begin{align}\label{1.h}
  h(u) = \sum_{i=1}^n h_{i}(u_i)
  + \frac{1}{2}\sum_{i,j=n+1}^N b_{ij}u_i u_j,
 \end{align}
where $(b_{ij})\in\R^{(N-n)\times(N-n)}$ is positive definite and
\begin{align*}
  h_{i}(u_i) = \begin{cases}
  u_i(\log u_i-1) & \mbox{if }m_i=1, \\
  \displaystyle\frac{u_i^{m_i}-u_i}{m_i(m_i-1)}
  &\mbox{if }1<m_i<2.
  \end{cases}
\end{align*} 
The first term in \eqref{1.h} is the Boltzmann entropy density ($m_i=1$) or Tsallis entropy density ($m_i>1$), while the second term is the Rao entropy density, which describes the mixing of the species. Although some of our results apply to general strongly convex functions (see Theorem \ref{thm.wsu4} below), we focus on the class of entropies \eqref{1.h} to simplify the presentation and computations. System \eqref{1.eq}--\eqref{1.bic} has an entropy structure if $h''(u)A(u)$ is positive semidefinite, where $h''(u)$ is the Hessian of $h$. This condition is natural in many systems from physics and biology; see \cite[Chap.~4]{Jue16} and Section \ref{sec.ex}. 

The diffusion coefficients are supposed to be given by
\begin{align*}
  A_{ij}(u) = a_i(u)\delta_{ij} + a_{ij}(u), \quad i,j=1,\ldots,N,
\end{align*}
where $a_i(u)$ is strongly positive, both $a_i$ and $a_{ij}$ are Lipschitz continuous, and $a_{ij}(u)$ behaves like $u_i^{2-m_i}$. The latter condition removes singularities coming from the product $h_i''(u)a_{ij}(u)$. In general, we require more than just positive semidefiniteness of $h''(u)A(u)$ to derive suitable gradient bounds. We consider here the case that $h''(u)A(u)$ is positive definite or it is positive definite up to some ``defect''.

In the literature, the strong solution $v=(v_1,\ldots,v_n)$ is usually assumed to be strictly positive and to satisfy the regularity condition $v_i\in L^\infty(0,T;W^{1,\infty}(\Omega))$ \cite{ChJu19,HeJu26} or $v_i\in C^{0,1}(\overline{\Omega}\times[0,T])$ \cite{HoBu22}. The novel contributions of this paper are twofold:
\begin{itemize}
\item The strong solution $v_i$ only needs to be nonnegative and to satisfy $v_i\in L^2(0,T;$ $W^{1,p}(\Omega))$ with $p>d$.
\item The matrix $h''(u)A(u)$ may be not positive definite but it fulfills a suitable hypocoercivity condition.
\end{itemize}

We specify our notion of weak and strong solution to \eqref{1.eq}--\eqref{1.bic}. Let $T>0$ and set $\Omega_T:=\Omega\times(0,T)$. Let $\mathcal{D}$ be a bounded convex domain. 

\begin{definition}[Bounded weak solution]\label{def.weak}
We call $u=(u_1,\ldots,u_N)$ a {\rm bounded weak solution} to \eqref{1.eq}--\eqref{1.bic} if $u(x,t)\in\overline{\mathcal{D}}$ for a.e.\ $(x,t)\in\Omega_T$,
\begin{align*}
  u_i\in L^2(0,T;H^1(\Omega)), \quad \pa_t u_i\in L^2(0,T;H^1(\Omega)'),
\end{align*}
and it holds for all $\phi_i\in L^2(0,T;H^{1}(\Omega))$ and $i=1,\ldots,n$ that
\begin{align*}
  \int_0^T\langle\pa_t u_i,\phi_i\rangle dt
  + \sum_{j=1}^N\int_0^T\int_\Omega A_{ij}(u)\na u_j\cdot\na\phi_i dxdt
  = \int_0^T\int_\Omega r_i(u)\phi_i dxdt.
\end{align*}
\end{definition}

The domain can be chosen as $\mathcal D = (0,L)^N$ for some $L>0$ (non-volume-filling case) or $\mathcal D = \{(u_1,\ldots,u_n)\in(0,1)^n :\sum_{i=1}^nu_i=1\}\times(0,L)^{N-n}$ (volume-filling case). It follows from the boundedness of the domain $\mathcal{D}$ that $u_i\in L^\infty(\Omega_T)$. In particular, it holds that $A_{ij}(u)\na u_j\in L^2(\Omega_T)$ such that the weak formulation of $u_i$ is well-defined. The existence of global-in-time bounded weak solutions to \eqref{1.eq}--\eqref{1.bic} can be shown for various models using the boundedness-by-entropy technique of \cite{BDPS10,Jue15}; see Section \ref{sec.ex} for details. The boundedness condition simplifies the technical proofs, since no growth conditions on the nonlinearities are required. The proof of the weak--strong uniqueness property can be shown without this condition for certain models using the cut-off technique of \cite[Sec.~3.2]{Fis17}, applied to cross-diffusion systems in \cite{ChJu19}. 

\begin{definition}[Strong solution]\label{def.str} 
A function $v=(v_1,\ldots,v_n)$ is called a {\em strong solution} to \eqref{1.eq}--\eqref{1.bic} if it is a bounded weak solution and satisfies
\begin{align*}
  \na v_i\in L^2(0,T;L^p(\Omega))\quad\mbox{for }i=1,\ldots,n,\ p>d.
\end{align*}
\end{definition}

The existence of global-in-time strong solutions to \eqref{1.eq}--\eqref{1.bic} is very delicate, and usually only the local-in-time existence can be proved, using the technique of Amann \cite{Ama93}. In fact, there exists a unique maximal classical solution $u$ to \eqref{1.eq}--\eqref{1.bic} on some time interval $[0,T^*)$ if the initial data satisfies $u^0\in W^{1,p}(\Omega;\R^n)$ for $p>d$ and the diffusion matrix $A(u)$ is positively stable, i.e., it possesses only eigenvalues with positive real part for all $u\in\mathcal{D}$. The latter condition is generally fulfilled for cross-diffusion systems with entropy structure \cite[Theorem 6]{ChJu21}; interestingly, the former condition corresponds to the regularity we impose on the strong solution. 

\subsection{Main results}

We impose the following general assumptions:
\begin{itemize}
\item[(A1)] Domains: $\Omega\subset\R^d$ ($d\ge 1$) with Lipschitz boundary $\pa\Omega$ and $\mathcal{D}\subset\R^N$ ($N\ge 1$) are bounded domains.
\item[(A2)] Initial data: $u^0=(u_1^0,\ldots,u_N^0)\in L^1(\Omega;\R^N)$ satisfies $u^0(x)\in\overline{\mathcal{D}}$ for a.e.\ $x\in\Omega$.
\item[(A3)] Diffusion matrix: $A_{ij}(u)=a_i(u)\delta_{ij}+a_{ij}(u)$, where $a_i$, $a_{ij}$ are Lipschitz continuous on $\overline{\mathcal{D}}$, and there exist $c_A>0$, $C_A>0$ such that $a_i(u)\ge c_A>0$ for $i=1,\ldots,n$, $a_i=0$ for $i=n+1,\ldots,N$, and $|a_{ij}(u)|\le C_A u_i^{2-m_i}$ for $u\in\overline{\mathcal{D}}$ and $i=1,\ldots,n$, $j=1,\ldots,N$, where $1\le m_i<2$.
\item[(A4)] Reaction rate: $r\in C^{0}(\overline{\mathcal{D}};\R^N)$ is bounded and Lipschitz continuous. 
\end{itemize}

We assume that the image of the weak solution to \eqref{1.eq}--\eqref{1.bic} is a subset of $\overline{\mathcal D}$, giving the boundedness of $u_i$. The strict positivity of $a_i(u)$ represent the non-degeneracy condition that we mentioned before. The condition $1\le m_i<2$ is needed to obtain the strong convexity of the glued entropy density uniformly for $u_i$ close to zero; see Lemma \ref{lem.convex} below.

We consider three different situations (see Table \ref{table} for a summary): 
\begin{itemize}
\item cross-diffusion systems without volume filling (which means that all unknowns $u_i$ are interpreted as densities);
\item cross-diffusion systems with volume filling in the first $n$ variables ($u_1,\ldots,u_n$ are volume fractions satisfying $\sum_{i=1}^nu_i=1$);
\item cross-diffusion systems with a lack of coercivity in the variable $u_1$ (corresponding to the volume fraction of water or the void, being diffusion-free).
\end{itemize}
The different situations are distinguished by the assumptions on the matrix $h''(u)A(u)$ discussed next.

\subsection*{Systems without volume filling}

Recall definition \eqref{1.h} of the entropy density $h$. In addition to Assumptions (A1)--(A4), we suppose that the matrix $h''(u)A(u)$ is positive definite:

\begin{itemize}
\item[(A5)$_1$] Coercivity: There exists $\lambda>0$ such that 
\begin{align*}
  z^T h''(u)A(u)z \ge \lambda|z|^2 \quad\mbox{for all }z\in\R^N,\
  u\in\mathcal{D}.
\end{align*} 
\end{itemize}

\begin{theorem}[Weak--strong uniqueness I]\label{thm.wsu1}
Let Assumptions (A1)--(A4) and (A5)$_1$ hold. Let $u$ be a bounded weak solution and $v$ be a strong solution to \eqref{1.eq}--\eqref{1.bic} in the sense of Definitions \ref{def.weak} and \ref{def.str}, respectively. If the initial conditions for $u$ and $v$ coincide then $u(t)=v(t)$ in $\Omega$ for $t>0$.  
\end{theorem}


\subsection*{Systems with volume filling}

Volume filling means that the first $n$ variables satisfy the constraint $\sum_{i=1}^n u_i=1$. In this situation, we can relax the positive definiteness of $h''(u)A(u)$:

\begin{itemize}
\item[(A5)$_2$] Hypocoercivity: There exist $\kappa\ge 0$ and $\lambda>0$ such that 
\begin{align*}
  z^T h''(u)A(u)z \ge \lambda|z|^2 
  - \kappa\bigg(\sum_{i=1}^n z_i\bigg)^2
  \quad\mbox{for all }z\in\R^N,\ u\in\mathcal{D}
\end{align*}
and $\sum_{i=1}^n A_{ij}(u)=0$ for all $j=1,\ldots,n$, $u\in\mathcal D$.
\item[(A6)$_2$] Reaction rate: It holds that $\sum_{i=1}^n r_i(u)=0$ for all $u\in\mathcal{D}$.
\end{itemize}

\begin{theorem}[Weak--strong uniqueness II]\label{thm.wsu2}
Let Assumptions (A1)--(A4), (A5)$_2$, and (A6)$_2$ hold. Let $u$ be a bounded weak solution and $v$ be a strong solution to \eqref{1.eq}--\eqref{1.bic} in the sense of Definitions \ref{def.weak} and \ref{def.str}, respectively, satisfying $\sum_{i=1}^n u_i=\sum_{i=1}^n v_i=1$ in $\Omega_T$. If the initial conditions for $u$ and $v$ coincide then $u(t)=v(t)$ in $\Omega$ for $t>0$.  
\end{theorem}


\subsection*{Systems with lack of coercivity in $u_1$}

When the variable $u_1$ represents water or space vacancies, there is no diffusion in that variable. Then $h''(u)A(u)$ cannot be coercive in $u_1$. Therefore, certain Lipschitz continuity conditions on the first reaction rate $r_1$ and the diffusion matrix are needed. Thus, we suppose the following assumptions:

\begin{itemize}
\item[(A5)$_3$] Hypocoercivity: There exist $\kappa\ge 0$ and $\lambda>0$ such that 
\begin{align*}
  z^T h''(u)A(u)z \ge \lambda\sum_{i=2}^N z_i^2
  - \kappa\bigg(\sum_{i=1}^n z_i\bigg)^2
  \quad\mbox{for all }z\in\R^N,\ u\in\mathcal{D}.
\end{align*}
Moreover, $u\mapsto u_1^{m_1-2}A_{1j}(u)$ is Lipschitz continuous in $\overline{\mathcal{D}}$ and $\sum_{i=1}^n A_{ij}(u)=0$ for all $j=1,\ldots,n$.
\item[(A6)$_3$] Reaction rate: It holds that $\sum_{i=1}^n r_i(u)=0$, $r_1(u)=0$ if $u_1=0$, and there exists $C>0$ such that for $u$, $v\in\mathcal{D}$,
\begin{align*}
  (r_1(u)-r_1(v))(h_1'(u_1)-h_1'(v_1))
  \le C\sum_{i=1}^N(u_i-v_i)^2.
\end{align*}
\end{itemize}

The identity $\sum_{i=1}^n r_i(u)=0$ is a compatibility condition to achieve mass conservation for the first $n$ variables. Indeed, summing \eqref{1.eq} over $i=1,\ldots,n$ and invoking both the volume-filling constraint as well as the assumption $\sum_{i=1}^n A_{ij}(u)=0$ from Assumption (A5)$_2$ shows that the condition $\sum_{i=1}^n r_i(u)=0$ is necessary. The condition $r_1(u)=0$ if $u_1=0$ ensures the nonnegativity of this variable. Finally, the inequality in Assumption (A6) is a natural requirement to estimate the reaction terms; see, e.g., Hypothesis (H3) in \cite{Jue15}.

\begin{theorem}[Weak--strong uniqueness III]\label{thm.wsu3}
Let Assumptions (A1)--(A4), (A5)$_3$, and (A6)$_3$ hold. Let $u$ be a bounded weak solution and $v$ be a strong solution to \eqref{1.eq}--\eqref{1.bic} in the sense of Definitions \ref{def.weak} and \ref{def.str}, respectively. If the initial conditions for $u$ and $v$ coincide then $u(t)=v(t)$ in $\Omega$ for $t>0$.
\end{theorem}


\begin{table}[ht]
\caption{Overview of various models for which the weak--strong uniqueness property for bounded weak solutions can be proved; see Section \ref{sec.ex}. The improvement consists in relaxing the regularity assumptions on the strong solution. A uniqueness result for the granular segregation model under a smallness condition was shown in \cite{GJV03}; see Section \ref{sec.granu}.}
\begin{tabular}{|l|p{55mm}|l|}
\hline 
Model & Weak--strong uniqueness & Thm. \\ \hline 
Regularized Busenberg--Travis & improves \cite{LaMa23} & 1 \\
SKT population & improves \cite{ChJu19} & 1 \\
Semiconductors with electron--hole scattering & new & 1 \\
Chemotaxis with additional cross-diffusion & new & 1 \\
Granular segregation & new & 2 \\
Maxwell--Stefan & improves \cite{HJT22} & 2 \\
Solar-cell thin-film & improves \cite{HoBu22} & 2 \\
Multiphase flow & improves \cite{HeJu26} & 2 \\
Chemotaxis multiphase flow & new & 3 \\ \hline
\end{tabular}
\label{table}
\end{table}


\subsection{Key ideas of the proofs}

The proofs of Theorems \ref{thm.wsu1}--\ref{thm.wsu3} are based on the relative entropy method, using the relative entropy density
\begin{align*}
  h(u|v) = h(u) - h(v) - h'(v)\cdot(u-v) \quad\mbox{for }u,v\in 
  \mathcal{D},
\end{align*}
recalling definition \eqref{1.h} of the entropy density $h$. To illustrate the idea, we first suppose that Assumption (A5)$_1$ holds. The time derivative along two positive smooth solutions becomes, after a computation detailed in Theorem \ref{thm.wsu4},
\begin{align}\label{1.I15}
  & \frac{d}{dt}\int_\Omega h(u|v)dx = I_1+\cdots+I_5, 
  \quad\mbox{where} \\
  & I_1 = -\sum_{i,j=1}^N\int_\Omega B_{ij}(u)
  \na\bigg(\frac{\pa h}{\pa u_i}(u)-\frac{\pa h}{\pa u_i}(v)\bigg)
  \cdot\na\bigg(\frac{\pa h}{\pa u_j}(u)-\frac{\pa h}{\pa u_j}(v)
  \bigg)dx, \nonumber \\
  & I_2 = -\sum_{i,j=1}^N\int_\Omega B_{ij}(v)
  \na\bigg(\frac{\pa h}{\pa u_i}(u) - \frac{\pa h}{\pa u_i}(v)
  - \sum_{k=1}^N\frac{\pa^2 h}{\pa u_i\pa u_k}(v)(u_k-v_k)\bigg)
  \cdot\na\frac{\pa h}{\pa u_j}(v)dx, \nonumber \\
  & I_3 = -\sum_{i,j=1}^N\int_\Omega(B_{ij}(u)-B_{ij}(v))
  \na\bigg(\frac{\pa h}{\pa u_i}(u) - \frac{\pa h}{\pa u_i}(v)\bigg)
  \cdot\na\frac{\pa h}{\pa u_j}(v)dx, \nonumber \\
  & I_4 = \sum_{i=1}^N\int_\Omega(r_i(u)-r_i(v))
  \bigg(\frac{\pa h}{\pa u_i}(u) - \frac{\pa h}{\pa u_i}(v)\bigg)dx, 
  \nonumber \\
  & I_5 = \sum_{i=1}^N\int_\Omega r_i(v)
  \bigg(\frac{\pa h}{\pa u_i}(u) - \frac{\pa h}{\pa u_i}(v)
  - \sum_{k=1}^N\frac{\pa^2 h}{\pa u_i\pa u_k}(v)(u_k-v_k)\bigg)dx,
  \nonumber 
\end{align}
and the mobility matrix $B(u) := A(u)h''(u)^{-1}$ is positive definite, since $h''(u)A(u)$ is positive definite. Therefore, after some calculations, we can estimate the first term $I_1$ as
\begin{align*}
  I_1 \le -\frac{\lambda}{2}\int_\Omega|\na(u-v)|^2 dx
  + C\int_\Omega|u-v|^2|\na v|^2 dx,
\end{align*}
where $\lambda>0$ comes from Assumption (A5)$_1$ and $C>0$ is some constant. We use Taylor expansion to bound $I_2$, while $I_3$ is estimated by using the Lipschitz continuity of $h''(u)$, leading after an estimation to
\begin{align*}
  I_2 + I_3 \le \frac{\lambda}{4}\int_\Omega|\na(u-v)|^2 dx
  + C(\lambda)\int_\Omega|u-v|^2|\na v|^2 dx.
\end{align*}
Finally, using Lipschitz continuity of the reaction rates, we obtain
\begin{align*}
  I_4 + I_5 \le C\int_\Omega|u-v|^2 dx. 
\end{align*}
Collecting these estimates, we conclude that
\begin{align*}
  \frac{d}{dt}\int_\Omega h(u|v)dx 
  \le -\frac{\lambda}{4}\int_\Omega|\na(u-v)|^2 dx
  + C\int_\Omega|u-v|^2 dx + C(\lambda)\int_\Omega|u-v|^2|\na v|^2 dx.
\end{align*}
Usually, the last term is estimated in terms of the $L^2(\Omega)$ difference $u-v$ assuming that $|\na v|$ is bounded. To refine this argument, we use the Gagliardo--Nirenberg inequality as in \cite{DeZa26} to obtain (see \eqref{2.GN} below)
\begin{align*}
  \int_\Omega|u-v|^2|\na v|^2 dx
  \le \frac{\lambda}{4}\int_\Omega|\na(u-v)|^2 dx
  + C(\lambda)\big(1+\|\na v\|_{L^p(\Omega)}^2\big)\int_\Omega|u-v|^2dx,
\end{align*}
where $p>d$. Thus, if $|\na v|$ is bounded in $L^p(\Omega)$, 
\begin{align*}
  \frac{d}{dt}\int_\Omega h(u|v)dx 
  \le C(\lambda,\na v)\int_\Omega|u-v|^2dx.
\end{align*}
Then, if $h(u|v)\ge C|u-v|^2$, we can apply Gronwall's inequality to conclude that $h(u(t)|v(t))$ $\le h(u(0)|v(0))=0$, implying that $u(t)=v(t)$ in $\Omega$ for $t>0$.

The inequality $h(u|v)\ge C|u-v|^2$ follows from the convexity of the entropy density if $v_i$ is strongly positive. Indeed, the derivative $(\pa^2 h_i/\pa u_i\pa u_j)(v)$ behaves like $v_i^{m_i-2}$ for $i=1,\ldots,n$, which diverges as $v_i\to 0$. To relax the positivity condition, we introduce a regularized entropy, the so-called ``glued entropy'' $h_\eps(u)$ for $\eps>0$. It is defined by $h''_{\eps,i}(u)=\eps^{m_i-2}$ if $u_i<\eps$, $h_{\eps,i}''(u)=h''(u)$ if $u_i>2\eps$, and smooth else, where $i=1,\ldots,n$; see Section \ref{sec.glued} for the precise definition. This construction was first suggested in \cite{BRZ22} to prove the partial H\"older regularity for solutions to \eqref{1.eq}--\eqref{1.bic}. Its use for weak--strong uniqueness problems is new. The construction is simple yet nontrivial, as it avoids the regularization of the diffusion matrix required in \cite[Sec.~2.3]{CDJ18} while preserving positive definiteness. 

The previous arguments hold if $h''(u)A(u)$ is positive definite to estimate $I_1$. This covers not only the situation of Theorem \ref{thm.wsu1} but also of Theorem \ref{thm.wsu2}. Indeed,  Assumption (A5)$_2$ implies that $h''(u)A(u)$ is positive definite on the subspace $\{z\in\R^N:\sum_{i=1}^n z_i=0\}$, and $\pa u/\pa x_j$ is an element of that space, thanks to the volume-filling constraint $\sum_{i=1}^n\na u_i=0$. In fact, we can remove the variable $u_1$ by using $u_1 = 1-\sum_{i=2}^n u_i$ and work with the remaining $N-1$ variables. Therefore, we can unify the proofs of Theorem \ref{thm.wsu1} (for $N$ variables) and Theorem \ref{thm.wsu2} (for $N-1$ variables); see the proof of Theorem \ref{thm.wsu4}. 

The situation of Theorem \ref{thm.wsu3} is different, since we do not have coercivity in the variable $u_1$. However, we may again express $u_1 = 1-\sum_{i=2}^n u_i$ in terms of the remaining $N-1$ variables. Still, the lack of coercivity poses some technical difficulties, requiring careful estimations.

The paper is organized as follows. Section \ref{sec.thm12} is devoted to the construction of the glued entropy and the proofs of Theorems \ref{thm.wsu1} and \ref{thm.wsu2}, whereas the proof of Theorem \ref{thm.wsu3} is presented in Section \ref{sec.thm3}. In Section \ref{sec.ex}, we present several examples of cross-diffusion systems to which Theorems \ref{thm.wsu1}--\ref{thm.wsu3} apply.


\section{Proof of Theorems \ref{thm.wsu1} and \ref{thm.wsu2}}
\label{sec.thm12}

We first detail the construction and some properties of the glued entropy, explain the reduction from $N$ to $N-1$ variables in the situation of Theorem \ref{thm.wsu2}, and then prove a weak--strong uniqueness result that includes both Theorems \ref{thm.wsu1} and \ref{thm.wsu2}. 

\subsection{Construction of the glued entropies}\label{sec.glued}

We regularize the entropy density \eqref{1.h} in such a way that the second derivative is bounded on $\overline{\mathcal{D}}$. Let $\eps>0$. The regularized entropy density is defined by
\begin{align*}
  h_\eps(u) = \sum_{i=1}^n h_{\eps,i}(u_i) + \frac12\sum_{i,j=n+1}^N
  b_{ij}u_iu_j,
\end{align*}
where the {\em glued entropy} is given by \cite[Sec.~3.2]{BRZ22}
\begin{align}\label{2.heps}
  h_{\eps,i}(u_i) = \int_0^{u_i}\int_0^z h_i''(\eps\eta_\eps^1(y)
  + y\eta_\eps^2(y))dydz,
\end{align}
where $(\eta_\eps^1,\eta_\eps^2)$ is a partition of unity of $\R$ such that $0\le\eta_\eps^j\le 1$ in $\R$, $\eta_\eps^1=1$ in $(-\infty,\eps)$, $\eta_\eps^2=1$ in $(2\eps,\infty)$, and $\sum_{j=1}^2\eta_\eps^j=1$ in $\R$. In particular, we have (see Figure \ref{fig.h2})
\begin{align}\label{2.hpp}
  h_{\eps,i}''(u_i) = \begin{cases}
  \eps^{m_i-2} & \mbox{if }0\le u_i < \eps, \\
  \mbox{smooth, decreasing} &\mbox{if }\eps\le u_i\le 2\eps, \\
  u_i^{m_i-2} &\mbox{if }u_i>2\eps.
  \end{cases}
\end{align}
The Hessian of $h_\eps(u)$ can be written as the block matrix
\begin{align*}
  h_\eps''(u) = \begin{pmatrix}
  \operatorname{diag}(h_{\eps,1}''(u_1),\ldots,h_{\eps,n}''(u_n))
  & 0 \\ 0 & b_{ij} \end{pmatrix}.
\end{align*}

\begin{figure}
    \centering
    \includegraphics[width=75mm]{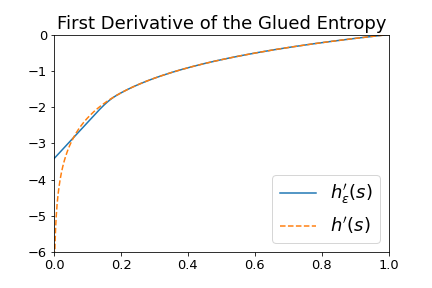}
    \includegraphics[width=75mm]{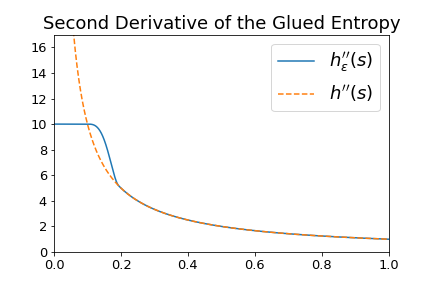}
    \caption{First (left) and second (right) derivative $h_\eps''$ of the glued entropy $h_\eps$ associated to $h(u) = u(\log u-1)$, where $\eps=0.1$.}
    \label{fig.h2}
\end{figure}

The glued entropy has two important properties: It is strongly convex, and $h_\eps''(u)A(u)$ is positive definite uniformly in $u\in\overline{\mathcal{D}}$. In the following, these properties are proved.

\begin{lemma}\label{lem.convex}
The Hessian of the glued entropy density $h_\eps$, defined in \eqref{2.heps}, is strongly convex, i.e., there exists $c>0$ such that $z^Th_\eps''(u)z\ge c|z|^2$ for all $z\in\R^N$ and $u\in\overline{\mathcal{D}}$. 
\end{lemma}

\begin{proof}
We compute for $z\in\R^N$ and $u\in\overline{\mathcal{D}}$:
\begin{align*}
  z^Th''_\eps(u)z = \sum_{i=1}^n h_{i,\eps}''(u_i)z_i^2
  + \sum_{i,j=n+1}^N b_{ij}z_iz_j.
\end{align*}
By construction of the glued entropy density and because of $m_i<2$, we can find for any $\gamma>0$ a value $\eps_\gamma>0$ such that
\begin{align}\label{2.hkappa}
  \inf_{0<u_i\le 2\eps_\gamma}|h''_{i,\eps}(u_i)|
  \ge (2\eps_\gamma)^{m_i-2}\ge \gamma.
\end{align}
Then the result follows from the positive definiteness of $(b_{ij})$. 
\end{proof}

\begin{lemma}\label{lem.posdef}
Let $h_\eps$ be the glued entropy density defined in \eqref{2.heps}. There exists $c>0$ such that
\begin{align*}
  z^T h_\eps''(u)A(u)z \ge c|z|^2 \quad\mbox{for }z\in\R^N,\  u\in\overline{\mathcal{D}}.
\end{align*}
In the volume-filling setting $\sum_{i=1}^n u_i=1$, this inequality holds for all $z\in\R^N$ satisfying $\sum_{i=1}^n z_i=0$. 
\end{lemma}

\begin{proof}
The proof is similar to \cite[Sec.~4]{BRZ22}, but since our entropy density is different, we detail the proof. The matrix $h''_\eps(u)A(u)$ can be written as the block matrix
\begin{align}\label{2.HepsA}
  h_\eps''(u)A(u) = \begin{pmatrix}
  Q^{11} & Q^{12} \\ Q^{21} & Q^{22} \end{pmatrix},
\end{align}
where the submatrices are given by 
\begin{align*}
  Q^{11}_{ij} &= (a_i(u)\delta_{ij} + a_{ij}(u))h''_{\eps,i}(u_i)
  &&\mbox{if }i,j=1,\ldots,n, \\ 
  Q^{12}_{ij} &= a_{ij}(u)h''_{\eps,i}(u_i) 
  &&\mbox{if }i=1,\ldots,n,\ j=n+1,\ldots,N, \\
  Q^{21}_{ij} &= \sum_{k=n+1}^N b_{ik}a_{kj}(u)
  &&\mbox{if }i=n+1,\ldots,N,\ j=1,\ldots,n, \\
  Q^{22}_{ij} &= \sum_{k=n+1}^N b_{ik}a_{kj}(u)
  &&\mbox{if }i,j=n+1,\ldots,N.
\end{align*}
It follows from \eqref{2.hkappa} that $h''_{\eps,i}(u_i)$ can be estimated for small values of $u_i$. For some fixed $u_i$, we introduce for $0<\eps<\eps_\gamma$ the set of indices
\begin{align*}
  S_\eps = \{i\in\{1,\ldots,n\}:u_i>2\eps\}\cup\{n+1,\ldots,N\}.
\end{align*}
On this set, we have $h''_{\eps,i}(u_i)=h_i''(u_i)$, since either $u_i>2\eps$ (so no regularization is active) or $i=n+1,\ldots,N$, where the entropy is quadratic and hence requires no regularization. Splitting the set $\{1,\ldots,N\}^2$ into the sets $S_\eps\times S_\eps$, $S_\eps\times S_\eps^c$, $S_\eps^c\times S_\eps$, and $S_\eps^c\times S_\eps^c$ and rearranging the sets, we find for $z\in\R^N$ that
\begin{align}\label{2.aux}
  z^T h''_\eps(u)A(u)z = \sum_{i,j=1}^N z_i(h''_\eps(u)A(u))_{ij}z_j
  = J_1 + J_2 + J_3 + J_4,
\end{align}
where 
\begin{align*}
  J_1 &= \sum_{i,j\in S_\eps} z_i h''_{\eps,i}(u_i)A_{ij}(u)z_j,
  && J_2 = \sum_{i=n+1}^N\sum_{\substack{j=1, \\ u_j\le 2\eps}}^n
  z_i(h_\eps''(u)A(u))_{ij}z_j, \\
  J_3 &= \sum_{\substack{i=1, \\ u_i > 2\eps}}^n
  \sum_{\substack{j=1, \\ u_j\le 2\eps}}^n
  z_ih''_{\eps,i}(u_i)A_{ij}(u)z_j, 
  && J_4 = \sum_{\substack{i=1, \\ u_i\le 2\eps}}^n\sum_{j=1}^N
  z_i (h''_\eps(u)A(u))_{ij}z_j,
\end{align*}
and the sum $J_4$ includes both cases $u_j>2\eps$ and $u_j\le 2\eps$. According to Assumption (A5)$_1$, the first sum is estimated according to 
\begin{align*}
  J_1 \ge\lambda\sum_{i\in S_\eps}z_i^2.
\end{align*}
We conclude from the representation \eqref{2.HepsA} that the second sum $J_2$ only consists of the values $b_{ik}$ and $a_{kj}(u)$, which are bounded. Hence, 
\begin{align*}
  J_2 \ge -C\sum_{i=n+1}^N\sum_{\substack{j=1, \\ u_j\le 2\eps}}^n
  |z_iz_j|
  \ge -\frac{\lambda}{4}\sum_{i=n+1}^N z_i^2
  - C(\lambda)\sum_{\substack{j=1, \\ u_j\le 2\eps}}^n z_j^2,
\end{align*}
where the last step follows from Young's inequality. For the third sum $J_3$, we use the estimate $h''_{\eps,i}(u_i)\le Cu_i^{m_i-2}$, which follows from \eqref{2.hpp}. Then, together with the condition $|a_{ij}(u)|\le C_A u_i^{2-m_i}$, we obtain $h_{\eps,i}''(u_i)|a_{ij}(u)|\le C$. Therefore, 
\begin{align*}
  J_3 &= \sum_{\substack{i=1, \\ u_i > 2\eps}}^n
  \sum_{\substack{j=1, \\ u_j\le 2\eps}}^n
   z_i h''_{i}(u_i)(a_i(u)\delta_{ij}+a_{ij}(u))z_j 
  = \sum_{\substack{i=1, \\ u_i > 2\eps}}^n
  \sum_{\substack{j=1, \\ u_j\le 2\eps}}^n h''_{i}(u_i)a_{ij}(u)z_iz_j \\
  &\ge -C\sum_{\substack{i=1, \\ u_i > 2\eps}}^n
  \sum_{\substack{j=1, \\ u_j\le 2\eps}}^n |z_iz_j|
  \ge -\frac{\lambda}{4}\sum_{\substack{i=1, \\ u_i > 2\eps}}^n z_i^2
  - C(\lambda)\sum_{\substack{j=1, \\ u_j\le 2\eps}}^n z_j^2.
\end{align*}
The explicit structure of $A_{ij}(u)$ yields two terms in $J_4$:
\begin{align*}
  J_4 = \sum_{\substack{i=1, \\ u_i\le 2\eps}}^n 
  h''_{\eps,i}(u_i)a_i(u)z_i^2
  + \sum_{\substack{i=1, \\ u_i\le 2\eps}}^n\sum_{j=1}^N
  z_i h_{\eps,i}''(u_i)a_{ij}(u)z_j.
\end{align*}
We infer from \eqref{2.hkappa} and Assumption (A3) that the first part can be bounded by
\begin{align*}
  \sum_{\substack{i=1, \\ u_i\le 2\eps}}^n 
  h''_{\eps,i}(u_i)a_i(u)z_i^2 
  \ge c_A\gamma \sum_{\substack{i=1, \\ u_i\le 2\eps}}^n z_i^2.
\end{align*}
For the second part of $J_4$, we estimate as in $J_3$ to conclude that
\begin{align*}
  J_4 \ge c_A\gamma \sum_{\substack{i=1, \\ u_i\le 2\eps}}^n z_i^2
  - C\sum_{\substack{i=1, \\ u_i\le 2\eps}}^n\sum_{j=1}^N |z_iz_j|
  \ge c_A\gamma \sum_{\substack{i=1, \\ u_i\le 2\eps}}^n z_i^2
  - \frac{\lambda}{4}\sum_{j=1}^N z_j^2
  - C(\lambda)\sum_{\substack{i=1, \\ u_i\le 2\eps}}^n z_i^2.
\end{align*}
Summarizing these estimates, we infer from \eqref{2.aux} that
\begin{align*}
   z^T h''_\eps(u)A(u)z
  \ge \lambda\sum_{i\in S_\eps}z_i^2
  + (c_A\gamma - C(\lambda))
  \sum_{\substack{i=1, \\ u_i\le 2\eps}}^n z_i^2
  - \frac{\lambda}{4}\bigg(
  \sum_{i=n+1}^N z_i^2 
  + \sum_{\substack{i=1, \\ u_i > 2\eps}}^n z_i^2
  + \sum_{i=1}^N z_i^2\bigg).
\end{align*}
We choose $\gamma>0$ sufficiently large (therefore $\eps_\gamma$ becomes small) such that $c_A\gamma - C(\lambda)\ge \lambda/2$. This gives
\begin{align*}
  z^T h''_\eps(u)A(u)z
  \ge \lambda\sum_{i\in S_\eps}z_i^2
  + \frac{\lambda}{2}\sum_{\substack{i=1, \\ u_i\le 2\eps}}^n z_i^2
  - \frac{\lambda}{4}\bigg(\sum_{i=n+1}^N z_i^2 
  + \sum_{\substack{i=1, \\ u_i > 2\eps}}^n z_i^2
  + \sum_{i\in S_\eps}z_i^2
  + \sum_{i\in S_\eps^c} z_i^2\bigg).
\end{align*}
As $S_\eps$ contains the indices $i=1,\ldots,n$ such that $u_i>2\eps$ or $i=n+1,\ldots,N$ and $S_\eps^c$ contains the indices $i=1,\ldots,n$ such that $u_i\le 2\eps$, we have
\begin{align*}
  z^T h''_\eps(u)A(u)z \ge \lambda\sum_{i\in S_\eps}z_i^2
  + \frac{\lambda}{4}\sum_{\substack{i=1, \\ u_i\le 2\eps}}^n z_i^2
  - \frac{\lambda}{2}\sum_{i\in S_\eps}z_i^2 
  \ge \frac{\lambda}{4}\sum_{i=1}^N z_i^2.
\end{align*}

Next, consider the volume-filling case. Let $z\in\R^N$ be such that $\sum_{i=1}^n z_i=0$. The estimation of $J_2$, $J_3$, and $J_4$ is as before, while Assumption (A5)$_2$ leads to
\begin{align*}
  J_1 &\ge \lambda\sum_{i\in S_\eps} z_i^2 
  - \kappa\bigg(\sum_{\substack{i=1, \\ u_i\in S_\eps}}^n z_i\bigg)^2
  = \lambda\sum_{i\in S_\eps} z_i^2 
  - \kappa\bigg(\sum_{\substack{i=1, \\ u_i\in S_\eps}}^n z_i\bigg)
  \bigg(\sum_{\substack{i=1, \\ u_i\in S_\eps}}^n z_i
  - \sum_{i=1}^n z_i\bigg) \\
  &= \lambda\sum_{i\in S_\eps} z_i^2 
  - \kappa\bigg(\sum_{\substack{i=1, \\ u_i\in S_\eps}}^n z_i\bigg)
  \bigg(\sum_{\substack{i=1, \\ u_i\le 2\eps}}^n z_i\bigg)
  \ge \lambda\sum_{i\in S_\eps} z_i^2 
  - \frac{\lambda}{4}\sum_{\substack{i=1, \\ u_i\in S_\eps}}^n z_i^2
  - C(\kappa,\lambda)\sum_{\substack{i=1, \\ u_i\le 2\eps}}^n z_i^2.
\end{align*}
We conclude that
\begin{align*}
  z^T h''_\eps(u)A(u)z &\ge \lambda\sum_{i\in S_\eps}z_i^2  
  - \big(c_A\gamma - C(\lambda) - C(\kappa,\lambda)\big)
  \sum_{\substack{i=1,\\  u_i\le 2\eps}}^n z_i^2 \\
  &\phantom{xx}- \frac{\lambda}{4}\bigg(\sum_{i=n+1}^N z_i^2 
  + \sum_{\substack{i=1, \\ u_i > 2\eps}}^n z_i^2
  + \sum_{i=1}^N z_i^2\bigg).
\end{align*}
From this point on, we can argue as in the previous setting to find that
\begin{align*}
  z^T h''_\eps(u)A(u)z \ge \frac{\lambda}{4}\sum_{i=1}^N z_i^2.
\end{align*}
This finishes the proof.
\end{proof}

In the volume-filling case, the variable $u_1 = 1-\sum_{i=2}^n u_i$ can be eliminated and expressed in terms of the reduced variable $\bar{u}=(u_2,\ldots,u_N)$. We show below that the quadratic form $z\mapsto z^T h''(u)A(u)z$ coincides with the quadratic form acting on the reduced variable. This result holds for any smooth function $h$ and any matrix $A$ with vanishing column sum. We now fix the set for admissible solutions as $\mathcal D = \{(u_1,\ldots,u_n)\in(0,1)^n :\sum_{i=1}^nu_i=1\}\times(0,L)^{N-n}\}$.

\begin{lemma}\label{lem.bar}
Let $h\in C^2(\overline{\mathcal{D}};\R)$ and $A=(A_{ij})\in\R^{N\times N}$ satisfying $\sum_{i=1}^n A_{ij}=0$ for $j=1,\ldots,n$. Define $\bar{h}:\R^{N-1}\to\R$ and the reduced matrix $\bar{A}\in\R^{(N-1)\times(N-1)}$ by, respectively,
\begin{align*}
  \bar{h}(\bar{u}) = h\bigg(1-\sum_{i=2}^n u_i,u_2,\ldots,u_N\bigg),
  \quad 
  \bar{A}_{ij} = \begin{cases}
  A_{ij} - A_{i1} &\mbox{if }i,j=2,\ldots,n, \\
  A_{ij} &\mbox{else},
  \end{cases}
\end{align*}
where $\bar{u}=(u_2,\ldots,u_N)\in \R^{N-1}$ is such that $(1- \sum_{i=2}^n u_i, u_2, \ldots, u_N) \in \overline{\mathcal{D}}$. For given $\bar{z}=(z_2,\ldots,z_N) \in\R^{N-1}$ define $z_1=-\sum_{i=2}^n z_i$ and $z=(z_1,\ldots,z_N)\in\R^N$. Then, for all $\bar{u}=(u_2,\ldots,u_N)$, where $u=(u_1,\ldots,u_N) \in \overline{\mathcal D}$ with $u_1=1-\sum_{i=2}^n u_i$, 
\begin{align*}
  \bar{z}^T\bar{h}''(\bar{u})\bar{A}\bar{z}  = z^T h''(u)Az.
\end{align*}
\end{lemma}

\begin{proof}
We introduce the matrix $M=(M_{ij})\in\R^{N\times(N-1)}$ by
\begin{align*}
  M_{ij} = \begin{cases}
  -1 &\mbox{if }i=1,\ j=2,\ldots,n, \\
  \delta_{ij} &\mbox{else},
  \end{cases}
\end{align*}
where $i=1,\ldots,N$ and $j=2,\ldots,N$. This matrix is the Jacobian of the mapping $(u_2,\ldots,u_N)\mapsto(1-\sum_{i=2}^n u_i,u_2,\ldots,u_N)$. Therefore, the chain rule yields the identity $\bar{h}''(\bar{u}) = M^T h''(u)M$. Let $\bar{z}=(z_2,\ldots,z_N) \in\R^{N-1}$ and set $z_1=-\sum_{i=2}^n z_i$. Then $M\bar{z}=z$. We compute for $i\in\{2,\ldots,n\}$,
\begin{align*}
  (\bar{A}\bar{z})_i &= \sum_{k=2}^n(A_{ik}-A_{i1})z_k 
  + \sum_{k=n+1}^N A_{ik}z_k
  = \sum_{k=2}^N A_{ik}z_k - A_{i1}\sum_{k=2}^n z_k
  = \sum_{k=1}^N A_{ik}z_k = (Az)_i.
\end{align*}
This shows that $(M\bar{A}\bar{z})_i = (\bar{A}\bar{z})_i = (Az)_i$ for $i=2,\ldots,n$. If $i=1$, we deduce from $A_{1j} = -\sum_{k=2}^n A_{kj}$ that
\begin{align*}
  (M\bar A\bar z)_1 = -\sum_{j=2}^n(\bar{A}\bar{z})_j 
  = -\sum_{j=2}^n(Az)_j = -\sum_{j=2}^n\sum_{k=1}^N A_{jk}z_k
  = \sum_{k=1}^N A_{1k}z_k = (Az)_1.
\end{align*}
Therefore, $M\bar A\bar z=Az$ and 
\begin{align*}
  \bar{z}^T\bar{h}''(\bar{u})\bar{A}\bar{z}
  = \bar{z}^T(M^T h''(u)M)\bar{A}\bar{z}
  = (M\bar{z})^T h''(u)(M\bar{A}\bar{z}) = z^Th''(u)Az,
\end{align*}
finishing the proof.
\end{proof}

With the glued entropy density $h_\eps$, we define the reduced glued entropy density
\begin{align}\label{2.reduced}
  \bar{h}_\eps(\bar{u}) 
  = h_\eps\bigg(1-\sum_{i=2}^n u_i,u_2,\ldots,u_N\bigg).
\end{align}
The previous lemmas imply that $\bar{h}_\eps$ is strongly convex and $\bar{h}''_\eps(\bar u)\bar A(\bar u)$ is positive definite uniformly in $u\in\overline{\mathcal{D}}$ also in the volume-filling setting. 

\begin{lemma}\label{lem.convex2}
The reduced glued entropy density $\bar{h}_\eps$ is strongly convex.
\end{lemma}

\begin{proof}
Let $M\in\R^{N\times(N-1)}$ be the matrix of the previous proof satisfying $\bar{h}''_\eps(\bar u)=M^Th_\eps''(u)M$. Then for $\bar{z}=(z_2,\ldots,z_N)$, by Lemma \ref{lem.convex} and definition of $M$,
\begin{align*}
  \bar{z}^T \bar{h}''_\eps(\bar u)\bar{z}
  = (M\bar{z})^T h''_\eps(u)(M\bar z) \ge c|M\bar{z}|^2
  = c\sum_{i=2}^N z_i^2 + c\bigg|\sum_{j=1}^n z_j\bigg|^2
  \ge c\sum_{i=2}^N z_i^2 = c|\bar{z}|^2,
\end{align*}
finishing the proof.
\end{proof}

\begin{lemma}\label{lem.posdef2}
Let $h_\eps$ be the glued entropy density defined in \eqref{2.heps} and let $\bar{h}_\eps$ be the reduced glued entropy density introduced in \eqref{2.reduced}. Then there exists $c>0$ such that
\begin{align*}
  \bar{z}^T \bar{h}_\eps''(\bar u)\bar A(\bar u)\bar{z} \ge c|\bar{z}|^2 
  \quad\mbox{for }\bar{z}\in\R^{N-1},\  u\in\overline{\mathcal{D}}.
\end{align*}
\end{lemma}

\begin{proof}
Let $\bar{z}=(z_2,\ldots,z_N)\in\R^{N-1}$ and set $z_1=-\sum_{i=2}^n z_i$ and $z=(z_1,z_2,\ldots,z_N)$. Then, by Lemmas \ref{lem.posdef} and \ref{lem.bar},
\begin{align*}
  \bar{z}^T \bar{h}_\eps''(\bar u)\bar A(\bar u)\bar{z}
  = z^T h''_\eps(u)A(u)z \ge c|z|^2 \ge c|\bar{z}|^2,
\end{align*}
proving the positive definiteness.
\end{proof}


\subsection{Weak--strong uniqueness proof}
\label{sec.wsu12}

Both entropy densities $h_\eps$ and $\bar{h}_\eps$, constructed in the previous section, satisfy two properties: The Hessian of the entropy density and the product of the Hessian of the entropy density with the diffusion matrix are positive definite. We show that these properties are sufficient to establish a weak--strong uniqueness property. Accordingly, throughout this subsection, we assume a general entropy density satisfying the following assumptions:
\begin{itemize}
\item[(B1)] The entropy density $h\in C^4(\overline{\mathcal{D}})$, $\mathcal{D}\subset \mathbb R^N$, is bounded from below and strongly convex, i.e., there exists $c_1>0$ such that $z^Th''(u)z\ge c_1|z|^2$ for all $z\in\R^N$ and $u\in\overline{\mathcal{D}}$.
\item[(B2)] There exists $c_2>0$ such that $z^Th''(u)A(u)z\ge c_2|z|^2$ for all $z\in\R^N$ and $u\in\overline{\mathcal{D}}$.
\end{itemize}
Condition (B1) implies that there exist constants $C_*$, $C^*>0$ such that 
\begin{align}\label{3.h2}
  C_*|u-v|^2 \le h(u|v) \le C^*|u-v|^2\quad\mbox{for all }
  u,v\in\overline{\mathcal{D}},  
\end{align}
where $h(u|v) =  h(u) - h(v) - h'(v)\cdot(u-v)$ is the relative entropy density. Indeed, by Taylor expansion,
\begin{align*}
  \frac{c_1}{2}|u-v|^2\le 
  h(u|v) = \int_0^1\theta(u-v)^Th''(u+\theta(u-v))
  (u-v)d\theta \le C|u-v|^2,
\end{align*}
where the lower bound follows from the strong convexity and the upper bound is a consequence of the regularity of $h$. In the volume-filling case, we suppose Assumption (A6)$_2$.

\begin{theorem}[Weak--strong uniqueness for general entropy]
\label{thm.wsu4}
Let Assumptions (A1)--(A3) and (B1)--(B2) hold. Let  $A:\overline{\mathcal{D}}\to\R^{N\times N}$ be Lipschitz continuous and $r:\overline{\mathcal{D}}\to\R^N$ be bounded and Lipschitz continuous. Furthermore, let $u$ be a bounded weak solution and $v$ be a strong solution to 
\eqref{1.eq}--\eqref{1.bic} in the sense of Definitions \ref{def.weak} and \ref{def.str} with the same initial data. Then $u(t)=v(t)$ in $\Omega$ for $t>0$.  
\end{theorem}

Since the glued entropy densities $h_\eps$ and $\bar{h}_\eps$ satisfy conditions (B1) and (B2), this result proves Theorems \ref{thm.wsu1} and \ref{thm.wsu2}.

\begin{proof}
By the strong convexity of $h$, the Hessian $h''(u)$ is invertible and the spectral norm of the inverse is bounded from above by $1/c_1$ (see condition (B1)). Since $h''$ is Lipschitz continuous, its inverse is Lipschitz continuous too, as we have
\begin{align*}
  |(h''&(u))^{-1}-(h''(v))^{-1}|
  \le |(h''(u))^{-1}||I-h''(u)(h''(v))^{-1}| \\
  &\le \frac{1}{c_1}|I-h''(u)(h''(v))^{-1}|
  = \frac{1}{c_1}|(h''(u)-h''(v))(h''(v))^{-1}|
  \le C|u-v|^2.
\end{align*}

Next, we observe that the regularity of the entropy density allows us to use $(\pa h/\pa u_i)(u)$ as a test function in the weak formulation of \eqref{1.eq}, and the chain rule \cite[Lemma A.1]{LaMa23} (here we use the boundedness of $h''(u)$) gives
\begin{align*}
  \frac{d}{dt}\int_\Omega h(u)dx
  = \langle \pa_t u,h'(u)\rangle
  = -\int_\Omega \na u^T:h''(u)A(u)\na u dx 
  + \int_\Omega r(u)\cdot h'(u)dx.
\end{align*}
The same identity holds for the strong solution $v$. Therefore,
\begin{align*}
  \frac{d}{dt}\int_\Omega h(u|v)dx
  &= \frac{d}{dt}\int_\Omega h(u)dx - \frac{d}{dt}\int_\Omega h(v)dx \\
  &\phantom{xx}- \int_\Omega h''(v):(\pa_t v\otimes(u-v))dx
  + \langle\pa_t(u-v),h'(v)\rangle \\
  &= -\int_\Omega\na u^T:h''(u)A(u)\na u dx 
  + \int_\Omega\na v^T: h''(v)A(v)\na v dx \\
  &\phantom{xx}
  + \int_\Omega \big(r(u)\cdot h'(u)dx - r(v)\cdot h'(v)\big)dx\\
  &\phantom{xx}- \int_\Omega h''(v):(\pa_t v\otimes(u-v))dx
  + \langle\pa_t(u-v),h'(v)\rangle.
\end{align*}
We insert equations \eqref{1.eq}, satisfied for $u$ and $v$:
\begin{align*}
  \frac{d}{dt}&\int_\Omega h(u|v)dx
  = -\int_\Omega(\na h'(u))^T:B(u)\na h'(u)dx
  - \int_\Omega(\na h'(v))^T:B(v)\na h'(v)dx \\
  &\phantom{xx}+ \sum_{i,j,k=1}^N\int_\Omega\na\bigg(
  \frac{\pa^2 h}{\pa u_i\pa u_j}(v)(u_j-v_j)\bigg)
  A_{ik}(v)\na v_k dx \\
  &\phantom{xx}+ \sum_{i,j=1}^N\int_\Omega
  \na\bigg(\frac{\pa h}{\pa u_i}(v)\bigg)
  \big(A_{ij}(u)\na u_j - A_{ij}(v)\na v_j\big)dx \\
  &\phantom{xx}
  + \sum_{i=1}^N\int_\Omega\bigg(r_i(u)\frac{\pa h}{\pa u_i}(u)
  - r_i(v)\frac{\pa h}{\pa u_i}(v)
  - \sum_{j=1}^N r_i(v)\frac{\pa^2 h}{\pa u_i\pa u_j}(v)(u_j-v_j)
  \bigg)dx \\
  &\phantom{xx}- \sum_{i=1}^N\int_\Omega(r_i(u)-r_i(v))
  \frac{\pa h}{\pa u_i}(v)dx,
\end{align*}
recalling the definition $B(u)=A(u)h''(u)^{-1}$. We infer after some rearrangements that
\begin{align*}
  \frac{d}{dt}&\int_\Omega h(u|v)dx
  = -\sum_{i,j=1}^N\int_\Omega B_{ij}(u)
  \na\frac{\pa h}{\pa u_i}(u)\cdot
  \na\frac{\pa h}{\pa u_j}(u)dx \\
  &\phantom{xx}- \sum_{i,j=1}^N\int_\Omega B_{ij}(v)
  \na\frac{\pa h}{\pa u_i}(v)\cdot
  \na\frac{\pa h}{\pa u_j}(v)dx \\
  &\phantom{xx}+ \sum_{i,j,k=1}^N\int_\Omega B_{ij}(v)
  \na\bigg(\frac{\pa^2 h}{\pa u_i\pa u_k}(v)(u_k-v_k)\bigg)
  \cdot\na\frac{\pa h}{\pa u_j}(v)dx \\
  &\phantom{xx}+ \sum_{i,j=1}^N\int_\Omega
  \na\bigg(\frac{\pa h}{\pa u_i}(v)\bigg)\cdot
  \bigg(B_{ij}(u)\na\frac{\pa h}{\pa u_j}(u)
  - B_{ij}(v)\na\frac{\pa h}{\pa u_j}(v)\bigg)dx \\
  &\phantom{xx}+ \sum_{i=1}^N\int_\Omega(r_i(u)-r_i(v))
  \bigg(\frac{\pa h}{\pa u_i}(u) - \frac{\pa h}{\pa u_i}(v)\bigg)dx \\
  &\phantom{xx}+ \sum_{i=1}^N\int_\Omega r_i(v)
  \bigg(\frac{\pa h}{\pa u_i}(u) - \frac{\pa h}{\pa u_i}(v)
  - \sum_{k=1}^N\frac{\pa^2 h}{\pa u_i\pa u_k}(v)(u_k-v_k)\bigg)dx.
\end{align*}
The second term on the right-hand side cancels with a part of the fourth term. Adding and subtracting the terms
\begin{align*}
  & \sum_{i,j=1}^N\int_\Omega B_{ij}(u)\na\bigg(\frac{\pa h}{\pa u_i}(u)
  - \frac{\pa h}{\pa u_i}(v)\bigg)\cdot \nabla \left( \frac{\partial h}{\partial u_j}(v)\right)dx,\\
   & \sum_{i,j=1}^N\int_\Omega B_{ij}(v)\na\bigg(\frac{\pa h}{\pa u_i}(u)
  - \frac{\pa h}{\pa u_i}(v)\bigg)\cdot \nabla \left( \frac{\partial h}{\partial u_j}(v)\right)dx,
\end{align*}
we end up with
\begin{align}\label{3.I59}
  \frac{d}{dt}\int_\Omega h(u|v)dx = I_1+\cdots+I_5, 
\end{align}
where $I_1,\ldots,I_5$ are defined in \eqref{1.I15}. 

For the first term $I_1$, we exploit the positive definiteness property (B2) for $z\in\R^N$:
\begin{align*}
  z^TB(u)z = (h''(u)z)^T(h''(u)A(u))(h''(u)^{-1}z)
  \ge \lambda|h''(u)^{-1}z|^2,
\end{align*}
which yields for $H^*(u):=h''(u)^{-1}$ and $H^*(u)=(H^*_{ij})$ that 
\begin{align*}
  I_1\le -\lambda\sum_{i=1}^N\int_\Omega\bigg|
  \sum_{j=1}^N H^*_{ij}(u)\na
  \bigg(\frac{\pa h}{\pa u_j}(u)-\frac{\pa h}{\pa u_j}(v)\bigg)
  \bigg|^2dx.
\end{align*}
The right-hand side can be estimated further by observing that
\begin{align*}
  |\na&(u_i-v_i)|^2
  = \bigg|\sum_{j=1}^N\bigg\{H^*_{ij}(u)\na
  \bigg(\frac{\pa h}{\pa u_j}(u)-\frac{\pa h}{\pa u_j}(v)\bigg)
  + (H^*_{ij}(u)-H^*_{ij}(v))\na\frac{\pa h}{\pa u_i}(v)\bigg\}\bigg| \\
  &\le 2\bigg|\sum_{j=1}^NH^*_{ij}(u)\na
  \bigg(\frac{\pa h}{\pa u_j}(u)-\frac{\pa h}{\pa u_j}(v)\bigg)\bigg|^2
  + 2\bigg|\sum_{j=1}^N(H^*_{ij}(u)-H^*_{ij}(v))
  \na\frac{\pa h}{\pa u_i}(v)\bigg|^2,
\end{align*}
leading to
\begin{align*}
  I_1 \le -\frac{\lambda}{2}\int_\Omega|\na(u-v)|^2 dx
  + \lambda\sum_{i=1}^N\int_\Omega
  \bigg|\sum_{j=1}^N(H^*_{ij}(u)-H^*_{ij}(v))
  \na\frac{\pa h}{\pa u_i}(v)\bigg|^2 dx.
\end{align*}
The boundedness of $h''(v)$ implies that $|\na h'(v)|=|h''(v)\na v|\le C|\na v|$. Hence, because of the Lipschitz continuity of $H_{ij}$, 
\begin{align*}
  I_1 \le -\frac{\lambda}{2}\int_\Omega|\na(u-v)|^2 dx
  + C\int_\Omega |u-v|^2|\na v|^2 dx.
\end{align*}

We compute the gradient in the integral $I_2$:
\begin{align*}
  I_2 &= -\sum_{i,j=1}^N\int_\Omega B_{ij}(v)\bigg\{
  \sum_{\ell=1}^N\frac{\pa^2 h}{\pa u_i\pa u_\ell}(u)\na u_\ell
  - \sum_{\ell=1}^N\frac{\pa^2 h}{\pa u_i\pa u_\ell}(v)\na v_\ell \\
  &\phantom{xx}- \sum_{k,\ell=1}^N
  \frac{\pa^3 h}{\pa u_i\pa u_k\pa u_\ell}(v)
  (u_k-v_k)\na v_\ell - \sum_{k=1}^N\frac{\pa^2 h}{\pa u_i\pa u_k}(v)
  \na(u_k-v_k)\bigg\}\cdot\na\frac{\pa h}{\pa u_j}(v)dx \\
  &= -\sum_{i,j,\ell=1}^N\int_\Omega B_{ij}(v)\bigg\{
  \bigg(\frac{\pa^2 h}{\pa u_i\pa u_\ell}(u)
  - \frac{\pa^2 h}{\pa u_i\pa u_\ell}(v)\bigg)\na(u_\ell-v_\ell) \\
  &\phantom{xx}+ \bigg(\frac{\pa^2 h}{\pa u_i\pa u_\ell}(u)
  - \frac{\pa^2 h}{\pa u_i\pa u_\ell}(v)
  - \sum_{k=1}^N\frac{\pa^3 h}{\pa u_i\pa u_\ell\pa u_k}(v)
  (u_k-v_k)\bigg)\na v_\ell\bigg\}\cdot\na\frac{\pa h}{\pa u_j}(v)dx.
\end{align*}
The Taylor expansion
\begin{align*}
  \bigg|\frac{\pa^2 h}{\pa u_i\pa u_\ell}(u)
  - \frac{\pa^2 h}{\pa u_i\pa u_\ell}(v)
  - \sum_{k=1}^N\frac{\pa^3 h}{\pa u_i\pa u_k\pa u_\ell}(v)
  (u_k-v_k)\bigg| \le C|u-v|^2
\end{align*}
then leads to
\begin{align*}
  I_2 &\le C\int_\Omega|u-v||\na(u-v)||\na v|dx
  + C\int_\Omega|u-v|^2|\na v|^2 dx \\
  &\le \delta\int_\Omega|\na(u-v)|^2 dx 
  + C(\delta)\int_\Omega|u-v|^2|\na v|^2 dx.
\end{align*}
We estimate $I_3$ by means of the Young inequality with $\delta>0$:
\begin{align*}
  I_3 &\le \delta\sum_{i=1}^N\int_\Omega
  \bigg|\na\bigg(\frac{\pa h}{\pa u_i}(u)
  - \frac{\pa h}{\pa u_i}(v)\bigg)\bigg|^2 dx
  + C(\delta)\sum_{i,j=1}^n\int_\Omega|B_{ij}(u)-B_{ij}(v)|^2
  \bigg|\na\frac{\pa h}{\pa u_j}(v)\bigg|^2 dx.
\end{align*}
The first term on the right-hand side is estimated according to
\begin{align*}
  \delta&\sum_{i=1}^N\int_\Omega
  \bigg|\na\bigg(\frac{\pa h}{\pa u_i}(u)
  - \frac{\pa h}{\pa u_i}(v)\bigg)\bigg|^2 dx \\
  &= \delta\sum_{i=1}^N\int_\Omega\bigg|\sum_{k=1}^N\bigg(
  \frac{\pa^2 h}{\pa u_i\pa u_k}(u) - \frac{\pa^2 h}{\pa u_i\pa u_k}(v)\bigg)\na v_k + \sum_{k=1}^N
  \frac{\pa^2 h}{\pa u_i\pa u_k}(u)\na(u_k-v_k)\bigg|^2dx \\
  &\le \delta C\int_\Omega|u-v|^2|\na v|^2 dx
  + \delta C\int_\Omega|\na(u-v)|^2 dx.
\end{align*}
This yields, together with the Lipschitz continuity of $B_{ij}$, that
\begin{align*}
  I_3 \le \delta C\int_\Omega|\na(u-v)|^2 dx
  + C(\delta)\int_\Omega|u-v|^2|\na v|^2 dx.
\end{align*}
Finally, it follows from the boundedness and Lipschitz continuity assumptions on $r_i$ and Taylor expansion that 
\begin{align*}
  I_4+I_5 \le C\int_\Omega|u-v|^2 dx.
\end{align*}
Summarizing all estimations, we conclude from \eqref{3.I59} that
\begin{align*}
  \frac{d}{dt}\int_\Omega h(u|v)dx 
  &\le -\bigg(\frac{\lambda}{2} - \delta C\bigg)
  \int_\Omega|\na(u-v)|^2 dx \\
  &\phantom{xx}+ C\int_\Omega|u-v|^2 dx
  + C\int_\Omega|u-v|^2|\na v|^2 dx.
\end{align*}

The last term is estimated by using the Gagliardo--Nirenberg inequality similarly as in \cite[Sec.~5]{DeZa26}. We use H\"older's inequality with $2<q<2d/(d-2)$, $1/p+1/q=1/2$ (hence $p>d$) and then Gagliardo--Nirenberg's inequality with $\theta = d/2-d/q\in(0,1)$:
\begin{align}\label{2.GN}
  \int_\Omega|u-v|^2|\na v|^2 dx 
  &\le \bigg(\int_\Omega|u-v|^q dx\bigg)^{2/q}
  \bigg(\int_\Omega|\na v|^p dx\bigg)^{2/p} \\
  &\le C\big(\|\na(u-v)\|_{L^2(\Omega)}^{2\theta}
  \|u-v\|_{L^2(\Omega)}^{2(1-\theta)} + \|u-v\|_{L^2(\Omega)}^2\big)
  \|\na v\|_{L^p(\Omega)}^2 \nonumber \\
  &\le \delta\|\na(u-v)\|_{L^2(\Omega)}^2
  + C\|u-v\|_{L^2(\Omega)}^2\big(1 + \|\na v\|_{L^p(\Omega)}^2\big).
  \nonumber 
\end{align}
We infer from the choice $0<\delta\le\lambda/(2C)$ and inequality \eqref{3.h2} that
\begin{align*}
  \frac{d}{dt}\int_\Omega h(u|v)dx 
  &\le -\bigg(\frac{\lambda}{2} - \delta C\bigg)
  \|\na(u-v)\|_{L^2(\Omega)}^2
  + C\big(1 + \|\na v\|_{L^p(\Omega)}^2\big)\|u-v\|_{L^2(\Omega)}^2 \\
  &\le C\big(1 + \|\na v\|_{L^p(\Omega)}^2\big)
  \int_\Omega h(u|v)dx.
\end{align*}
Since $v$ is a strong solution, we have $v_i\in L^2(0,T;L^p(\Omega))$ from Definition \ref{def.str}. Thus, we can apply Gronwall's inequality to find that $h(u(t)|v(t))\le C(T) h(u(0)|v(0))=0$. By inequalities \eqref{3.h2}, this gives $u(t)=v(t)$ in $\Omega$ for $t>0$.
\end{proof}


\section{Proof of Theorem \ref{thm.wsu3}}\label{sec.thm3}

We construct the glued entropy in a slightly different way than in Section \ref{sec.glued}. Then we show the Lipschitz continuity of certain nonlinearities with factor $u_i^{m_i-2}$ and finally prove Theorem \ref{thm.wsu3}. 

\subsection{Construction of the glued entropy density}\label{sec.glued2}

We define here the glued entropy density as
\begin{align}\label{3.heps}
  h_\eps(u) = h_1(u_1) + \sum_{i=2}^n h_{\eps,i}(u_i)
  + \sum_{i,j=n+1}^N b_{ij}u_iu_j
\end{align}
for $u\in\mathcal{D}$, where $h_{\eps,i}$ is defined in \eqref{2.heps} and $\mathcal{D}$ is given in Assumption (A1). 

%
%

\begin{lemma}
Let $h_\eps$ be the glued entropy density defined in \eqref{3.heps}. There exists $c>0$ such that for all $z\in\R^N$ with $\sum_{i=1}^nz_i=0$, 
\begin{align*}
  z^Th_\eps''(u)A(u)z\ge c\sum_{i=2}^N z_i^2
  \quad\mbox{for all }u\in\mathcal{D}.
\end{align*}  
\end{lemma}

\begin{proof}
The proof is similar to that one of Lemma \ref{lem.posdef}, but since we treat the first component differently, we need to take into account additional terms originating from that component. For some fixed $u\in\mathcal{D}$, we introduce the set
\begin{align*}
  S_\eps = \{1\}\cup\{i\in\{2,\ldots,n\}:u_i\ge 2\eps\}
  \cup\{n+1,\ldots,N\}.
\end{align*}
Let $z\in\R^N$ be such that $\sum_{i=1}^nz_i=0$. Then $z^Th''_\eps(u)A(u)z = J_5+\cdots+J_{10}$, where
\begin{align*}
  J_5 &= \sum_{i,j\in S_\eps}z_ih_{\eps,i}''(u_i)A_{ij}(u)z_j, 
  && J_6 = \sum_{i=n+1}^N\sum_{\substack{j=2, \\ u_j\le 2\eps}}^n
  z_i(h_\eps''(u)A(u))_{ij}z_j, \\
  J_7 &= \sum_{\substack{i=2, \\ u_i > 2\eps}}^n
  \sum_{\substack{j=2, \\ u_j\le 2\eps}}^n
  z_ih''_{\eps,i}(u_i)A_{ij}(u)z_j, 
  && J_8 = \sum_{\substack{i=2, \\ u_i\le 2\eps}}^n\sum_{j=2}^N
  z_i (h''_\eps(u)A(u))_{ij}z_j, \\
  J_{9} &= \sum_{\substack{j=2, \\ u_j\le 2\eps}}^n
  z_1 (h''_{\eps,1}(u)A_{1j}(u)z_j,
  && J_{10} = \sum_{\substack{i=2, \\ u_i\le 2\eps}}^n
  z_i h''_{\eps,i}(u)A_{i1}(u)z_1.
\end{align*}
The terms $J_5,\ldots,J_8$ correspond to $J_1,\ldots,J_4$ from the proof of Lemma \ref{lem.posdef}, while the terms $J_{9}$ and $J_{10}$ take into account the first component. Therefore, estimating as in the proof of Lemma \ref{lem.posdef} but excluding the first component, we obtain 
\begin{align*}
  J_5 &\ge \lambda\sum_{i\in S_\eps\setminus\{1\}}z_i^2
  - \kappa\bigg(\sum_{\substack{i\in S_\eps,\\ 1\le i\le n}}
  z_i\bigg)^2, \\
  J_6 &\ge -\frac{\lambda}{8}\sum_{i=n+1}^N z_i^2
  - C(\lambda)\sum_{\substack{j=2, \\ u_j\le 2\eps}}^n z_j^2, \\
  J_7 &\ge -\frac{\lambda}{8}\sum_{\substack{i=2, \\ u_i>2\eps}}^n z_i^2
  - C(\lambda)\sum_{\substack{j=2, \\ u_j\le 2\eps}}^n z_j^2, \\
  J_8 &\ge c_A\gamma\sum_{\substack{i=2, \\ u_i\le 2\eps}}^n z_i^2
  - \frac{\lambda}{8}\sum_{j=2}^Nz_j^2
  - C(\lambda)\sum_{\substack{i=2, \\ u_i\le 2\eps}}^n z_i^2.
\end{align*} 
The remaining terms $J_{9}$ and $J_{10}$ are estimated similarly as $J_3$, using the property $|a_{ij}(u)|$ $\times h''_{\eps,i}(u_i)\le C$ that originates from the growth condition of $a_{ij}$ in Assumption (A4):
\begin{align*}
  J_{9}+J_{10} \ge -\sum_{\substack{j=2, \\ u_j\le 2\eps}}^n|z_1z_j|
  - \sum_{\substack{i=2, \\ u_i\le 2\eps}}^n|z_iz_1|
  \ge -\frac{\lambda}{8n}|z_1|^2 
  - C(\lambda)\sum_{\substack{i=2, \\ u_i\le 2\eps}}^n z_i^2.
\end{align*}
We combine these estimates to find that
\begin{align*}
  z^Th''_\eps(u)A(u)z
  &\ge \lambda\sum_{i\in S_\eps\setminus\{1\}}z_i^2
  - \kappa\bigg(\sum_{\substack{i\in S_\eps,\\ 1\le i\le n}}z_i\bigg)^2
  + (c_A\gamma - C(\lambda))
  \sum_{\substack{i=2, \\ u_i\le 2\eps}}^n z_i^2 \\
  &\phantom{xx}- \frac{\lambda}{8}\bigg(\sum_{i=n+1}^N z_i^2
  + \sum_{i=2}^N z_i^2 
  + \sum_{\substack{i=2, \\ u_i>2\eps}}^n z_i^2\bigg)
  - \frac{\lambda}{8n}z_1^2.
\end{align*}
Using $z_1=-\sum_{i=2}^n z_i$, the second and last terms on the right-hand side are estimated as
\begin{align*}
  & \bigg(\sum_{\substack{i\in S_\eps,\\ 1\le i\le n}}z_i\bigg)^2
  = \bigg(\sum_{\substack{i\in S_\eps,\\ 2\le i\le n}}z_i + z_1\bigg)^2
  = \bigg(\sum_{\substack{i=2,\\ u_i\ge 2\eps}}^n z_i
  - \sum_{i=2}^n z_i\bigg)^2
  = \bigg(\sum_{\substack{i=2,\\ u_i < 2\eps}}^n z_i\bigg)^2
  \le n\sum_{\substack{i=2,\\ u_i\le 2\eps}}^n z_i^2, \\
  &\mbox{and}\quad \frac{\lambda}{8n}z_1^2 
  = \frac{\lambda}{8n}\bigg(\sum_{i=2}^n z_i\bigg)^2
  \le \frac{\lambda}{8}\sum_{i=2}^n z_i^2.
\end{align*}
Choosing $\gamma>0$ sufficiently large such that $c_A\gamma-C(\lambda)-\kappa n\ge \lambda$, we infer that
\begin{align*}
  z^Th''_\eps(u)A(u)z
  &\ge \lambda\sum_{i\in S_\eps\setminus\{1\}}z_i^2
  + \big(c_A\gamma - C(\lambda) - \kappa n\big)
  \sum_{\substack{i=2, \\ u_i\le 2\eps}}^n z_i^2
  - \frac{\lambda}{2}\sum_{i=2}^N z_i^2 \\
  &\ge \lambda\sum_{i\in S_\eps\setminus\{1\}}z_i^2
  + \lambda\sum_{\substack{i=2, \\ u_i\le 2\eps}}^n z_i^2
  - \frac{\lambda}{2}\sum_{i=2}^N z_i^2
  = \frac{\lambda}{2}\sum_{i=2}^N z_i^2.
\end{align*}
This finishes the proof.
\end{proof}


\subsection{Reduction to $N-1$ variables}

As in the proof of Theorem \ref{thm.wsu2}, we eliminate the variable $u_1=1-\sum_{i=2}^n u_i$ and formulate the problem in the reduced variable $\bar{u}=(u_2,\ldots,u_N)$. For this, recall the definition
\begin{align*}
  \bar{h}_\eps(\bar{u}) = h_\eps\bigg(1-\sum_{i=2}^n u_i,
  u_2,\ldots,u_N\bigg) \quad\mbox{for }\bar{u}\in\overline{\mathcal{D}}
\end{align*}
and the reduced diffusion matrix $\bar{A}\in\R^{(N-1)\times(N-1)}$, defined in Lemma \ref{lem.posdef}. We define $H(u)$, $Q\in\R^{(N-1)\times(N-1)}$ by
\begin{align*}
  H_{ij}(u) := \frac{\pa^2 h_\eps}{\pa u_i\pa u_j}(u)
  \mbox{ for }i,j=2,\ldots,n, \quad
  Q = \begin{cases}
  1 &\mbox{if }i,j=2,\ldots,n, \\ 0 &\mbox{else}.
\end{cases}
\end{align*}
Furthermore, we set $\mathbb{I}\in\R^{N-1}$ with $\mathbb{I}_i=1$ if $i=2,\ldots,n$ and $\mathbb{I}_i=0$ for $i=n+1,\ldots,N$. With these definitions, we can write
\begin{align*}
  \bar{h}_\eps''(\bar{u}) = H(u) + u_1^{m_1-2}Q.
\end{align*}

\begin{lemma}
The matrices $\bar{H}^*(\bar{u})=\bar{h}_\eps''(\bar{u})^{-1}$ and $B(u) = \bar{A}(\bar{u})\bar{h}''_\eps(\bar{u})^{-1}$ as well as the mappings
\begin{align}\label{3.map}
  \bar{u}\mapsto u_1^{m_1-2}\sum_{k=2}^n
  \bar{H}^*_{kj}(\bar{u}), \quad
  \bar{u}\mapsto u_1^{m_1-2}\sum_{k=2}^nB_{kj}(u)
\end{align}
for $j=1,\ldots,N$ are Lipschitz continuous in $\overline{\mathcal{D}}$, with the restriction that for different arguments $\bar{u}\neq\bar{v}$ in the Lipschitz inequality, we require that $v_1\ge c>0$ for some $c>0$.
\end{lemma}

\begin{proof}
To prove the Lipschitz continuity of the first mapping in \eqref{3.map}, we observe that it is sufficient to show the Lipschitz continuity of
\begin{align*}
  \bar{H}^*(\bar{u})(u_1^{m_1-2}\mathbb{I})
  = \big(H(u)+u_1^{m_1-2}Q\big)^{-1}(u_1^{m_1-2}\mathbb{I}).
\end{align*}
The inverse can be computed explicitly:
\begin{align*}
  \big(H(u)+u_1^{m_1-2}Q\big)^{-1}(u_1^{m_1-2}\mathbb{I})
  = \frac{H(u)^{-1}\mathbb{I}}{u_1^{2-m_1}
  + \mathbb{I}^TH(u)^{-1}\mathbb{I}}.
\end{align*}
We claim that this expression is Lipschitz continuous. Indeed, the matrix $H(u)$ is positive definite uniformly in $u$. Therefore, the inverse $H(u)^{-1}$ is uniformly bounded and
\begin{align*}
  |H(u)^{-1}-H(v)^{-1}| &\le |H(u)^{-1}||I-H(u)H(v)^{-1}|| \\
  &\le C|(H(v)-H(u))H(v)^{-1}| \le C|u-v|.
\end{align*}
Let $u_1\ge 0$, $v_1\ge c>0$, and let $a>0$ be any number. Then, by a Taylor expansion and using the condition $m_i\le 2$,
\begin{align*}
  \bigg|\frac{1}{u_1^{2-m_1}+a}-\frac{1}{v_1^{2-m_1}+a}\bigg|
  &= \bigg|(2-m_1)\int_0^1\frac{(u_1+(v_1-u_1)\theta)^{1-m_1}}{
  [(u_1+(v_1-u_1)\theta)^{2-m_1} + a]^2}(v_1-u_1)d\theta\bigg| \\
  &\le C|u_1-v_1|\int_0^1\frac{(c\theta)^{1-m_1}}{a^2}d\theta
  \le C(a)|u_1-v_1|.
\end{align*}
As the inverse $H(u)^{-1}$ is uniformly positive definite, we can estimate $\mathbb{I}^TH(u)^{-1}\mathbb{I}\ge a>0$ for some $a>0$. Therefore,
\begin{align*}
  \bigg|&\frac{1}{u_1^{2-m_1} + \mathbb{I}^TH(u)^{-1}\mathbb{I}}
  - \frac{1}{v_1^{2-m_1} + \mathbb{I}^TH(v)^{-1}\mathbb{I}}\bigg| \\
  &\le \bigg|\frac{1}{u_1^{2-m_1} + \mathbb{I}^TH(u)^{-1}\mathbb{I}}
  - \frac{1}{u_1^{2-m_1} + \mathbb{I}^TH(v)^{-1}\mathbb{I}}\bigg| \\
  &\phantom{xx}
  + \bigg|\frac{1}{u_1^{2-m_1} + \mathbb{I}^TH(v)^{-1}\mathbb{I}}
  - \frac{1}{v_1^{2-m_1} + \mathbb{I}^TH(v)^{-1}\mathbb{I}}\bigg|
  \le C|u-v| + C|u_1-v_1|.
\end{align*}
This proves the claim and the Lipschitz continuity of the first mapping in \eqref{3.map}. 

Next, we consider the matrix $\bar H^*(\bar u)$. Let $v\in\R^{N-1}$. We write
\begin{align*}
  \bar H^*(\bar u)v &= (H(u)+u_1^{m_1-2}Q)^{-1}v \\
  &= (H(u)+u_1^{m_1-2}Q)^{-1}
  \big((H(u)+u_1^{m_1-2}Q) - u_1^{m_1-2}Q\big)H(u)^{-1}v \\
  &= H(u)^{-1}v - (H(u)+u_1^{m_1-2}Q)^{-1}(u_1^{m_1-2}QH(u)^{-1})v.
\end{align*}
We have proved in the first part that $u\mapsto u_1^{m_1-2}QH(u)^{-1}$ is Lipschitz continuous. Moreover, $u\mapsto H(u)^{-1}$ is Lipschitz continuous and $\bar H^*(\bar u)$ is bounded. Therefore, $u\mapsto \bar H^*(\bar u)$ is Lipschitz continuous. 

We verify the Lipschitz continuity of the second mapping in \eqref{3.map}. By definition of $B(u)$ and $\bar A(u)$ as well as $\sum_{i=2}^n A_{ik}(u) = - A_{1k}(u)$ from Assumption (A5)$_3$, we have
\begin{align*}
  u_1^{m_1-2}\sum_{k=2}^nB_{kj}(u)
  &= u_1^{m_1-2}\sum_{i=2}^n\sum_{k=2}^N A_{ik}(u)\bar H^*_{kj}(\bar u)
  - u_1^{m_1-2}\sum_{i,k=2}^n A_{i1}(u)\bar H^*_{kj}(\bar u) \\
  &= -u_1^{m_1-2}\sum_{k=2}^N A_{1k}(u)\bar H^*_{kj}(\bar u)
  + A_{11}(u)\bigg(u_1^{m_1-2}\sum_{k=2}^n\bar H^*_{kj}(\bar u)\bigg).
\end{align*}
Again by Assumption (A5)$_3$, the function $u\mapsto u_1^{m_1-2}A_{1k}(u)$ is Lipschitz continuous, and we have shown this property for the first mapping in \eqref{3.map}. Hence, the second mapping in \eqref{3.map} is Lipschitz continuous as well.
\end{proof}


\subsection{Weak--strong uniqueness proof} 

We use the relative entropy method already used in the proof of Theorems \ref{thm.wsu1} and \ref{thm.wsu2}. To simplify the presentation, we first consider the case $r(u)=0$, explaining the general case later. The idea is to use the test function $\pa \bar{h}_\eps/\pa u_i$ in the weak formulation of the equations satisfied by the weak solution $\bar{u}$ and the strong solution $\bar{v}$. Thanks to the glued entropy, this can be done for $i=2,\ldots,N$, but we cannot argue in the same way for $i=1$. Therefore, we need to regularize. We use the test function $h_1'(u_1+\eta)$ for $\eta>0$, giving
\begin{align*}
  \int_\Omega h_1(u_1(t)+\eta)dx 
  + \sum_{j=1}^N\int_0^t\int_\Omega (u_1+\eta)^{m_1-2}
  A_{1j}(u)\na u_1\cdot\na u_j dxds
  = \int_\Omega h_1(u_1^0+\eta)dx.
\end{align*} 
The condition $|u_1^{m_1-2}a_{1j}(u)|\le C$ from Assumption (A5)$_3$ gives a uniform bound for the last term on the left-hand side, and the $L^\infty(\Omega_T)$ bound for $u$ provides a bound for the remaining terms. Thus, we can perform the limit $\eta\to 0$ using dominated convergence. Taking into account the equations for $i=2,\ldots,N$, we obtain
\begin{align*}
  \int_\Omega h_\eps(u(t))dx
  + \sum_{i,j=1}^N\int_0^t\int_\Omega(h_\eps''(u)A(u))_{ij}
  \na u_i\cdot\na u_j dx = \int_\Omega h_\eps(u^0)dxds.
\end{align*}
Repeating the arguments of the proof of Theorem \ref{thm.wsu4} leading to \eqref{3.I59} but separating the component $u_1$, we find that
\begin{align}\label{3.I1015}
  \frac{d}{dt}\int_\Omega\bar h_\eps(u|v)dx = I_{6}+\cdots+I_{10}, 
\end{align}
where
\begin{align*}
  I_{6} & = -\sum_{i,j=2}^N\int_\Omega B_{ij}(u)
  \na\bigg(\frac{\pa\bar h_\eps}{\pa u_i}(\bar u)
  - \frac{\pa\bar h_\eps}{\pa u_i}(\bar v)\bigg)
  \cdot\na\bigg(\frac{\pa\bar h_\eps}{\pa u_j}(\bar u)
  - \frac{\pa\bar h_\eps}{\pa u_j}(\bar v)\bigg)dx \nonumber \\
  I_{7} &= -\sum_{i,j=2}^N\int_\Omega B_{ij}(v)
  \na\bigg(\frac{\pa\bar h_\eps}{\pa u_i}(\bar u) 
  - \frac{\pa\bar h_\eps}{\pa u_i}(\bar v)
  - \sum_{k=2}^N\frac{\pa^2\bar h_\eps}{\pa u_i\pa u_k}(\bar v)
  (u_k-v_k)\bigg)
  \cdot\na\frac{\pa\bar h_\eps}{\pa u_j}(\bar v)dx, \nonumber \\
  I_{8} &= -\sum_{i,j=2}^N\int_\Omega(B_{ij}(u)-B_{ij}(v))
  \na\bigg(\frac{\pa\bar h_\eps}{\pa u_i}(\bar u) 
  - \frac{\pa\bar h_\eps}{\pa u_i}(\bar v)\bigg)
  \cdot\na\frac{\pa\bar h_\eps}{\pa u_j}(\bar v)dx, \nonumber \\ 
  I_{9} &= \sum_{i=2}^n\sum_{j=2}^N\int_\Omega B_{ij}(u)
  \na(h_{\eps,1}'(u_1)-h_{\eps,1}'(v_1))
  \cdot\na\frac{\pa\bar h}{\pa u_j}(\bar v)dx, \\
  I_{10} &= \sum_{i=2}^n\sum_{j=2}^N\int_\Omega B_{ij}(v)
  \na\bigg(v_1^{m_1-2}\sum_{k=2}^n(u_k-v_k)\bigg)\cdot
  \na\frac{\pa\bar h_\eps}{\pa u_j}(\bar v)dx.
\end{align*}
The terms $I_6,\ldots,I_8$ correspond to $I_1,\ldots,I_3$ in the proof of Theorem \ref{thm.wsu4}, while $I_9$ and $I_{10}$ account for the first component. We estimate $I_{6}$ by using Lemma \ref{lem.posdef2}:
\begin{align*}
  I_{6} \le -\lambda\sum_{i=2}^N\int_\Omega
  \bigg|\sum_{j=2}^N\bar H^*_{ij}(\bar u)\na
  \bigg(\frac{\pa\bar h_\eps}{\pa u_j}(\bar u)
  - \frac{\pa\bar h_\eps}{\pa u_j}(\bar v)\bigg)\bigg|^2dx.
\end{align*}
This term is reformulated in a slightly different way compared to the term $I_1$ in the proof of Theorem \ref{thm.wsu4}. First, we estimate for $i=2,\ldots,N$:
\begin{align*}
  |\na(u_i-v_i)|^2 &= \big|\big\{\bar{h}_\eps''(u)^{-1}
  \na(\bar h_\eps'(\bar u) - \bar h_\eps'(\bar v))
  + \bar{h}_\eps''(\bar u)^{-1}
  (\bar h_\eps''(\bar v)-\bar{h}''_\eps(\bar u))\na v\big\}_i\big|^2 \\
  &\le 2\big|\big\{\bar h_\eps''(\bar u)^{-1} 
  \na(\bar h_\eps'(\bar u) - \bar h_\eps'(\bar v))\big\}_i\big|^2
  + 2\big|\big\{\bar{h}_\eps''(\bar u)^{-1}
  (\bar h_\eps''(\bar v)-\bar{h}_\eps(\bar u))\na v\big\}_i\big|^2.
\end{align*}
Replacing the first term on the right-hand side, we arrive at
\begin{align}\label{3.I6}
  I_{6} &\le -\frac{\lambda}{2}\sum_{i=2}^N\int_\Omega
  |\na(u_i-v_i)|^2 dx \\
  &\phantom{xx}+ \lambda\sum_{i=2}^N\int_\Omega
  \bigg|\sum_{j,k=2}^N\bar H_{ij}^*(\bar u)\bigg(
  \frac{\pa\bar h_\eps}{\pa u_j\pa u_k}(\bar u)
  - \frac{\pa\bar h_\eps}{\pa u_j\pa u_k}(\bar v)\bigg)
  \na v_k\bigg|dx \nonumber \\
  &= -\frac{\lambda}{2}\sum_{i=2}^N\int_\Omega
  |\na(u_i-v_i)|^2 dx + \lambda \int_\Omega
  |\bar H^*(\bar u)(H(u)-H(v))\na v|dx \nonumber \\
  &\phantom{xx}+ \lambda \int_\Omega
  |(u_1^{m_1-2}-v_1^{m_1-2})\bar H^*(\bar u)Q\na v|dx. \nonumber 
\end{align}
The matrix $\bar h_\eps''(\bar u)=H(u)+u_1^{m_1-2}Q$ is bounded and positive definite uniformly in $u$. Therefore, its inverse $\bar H^*(\bar u)$ is bounded too, which leads to $|\bar H^*(\bar u)(H(u)-H(v))|\le \sum_{i=2}^N|u_i-v_i|$. We write the integrand of the last term in \eqref{3.I6} as the solution $W\in\R^{(N-1)\times(N-1)}$ to the equation
\begin{align*}
  (u_1^{m_1-2}-v_1^{m_1-2})Q = (H(u)+u_1^{m_1-2}Q)W.
\end{align*}
Then $W$ can be written as $W = (n-1)^{-1} u_1^{2-m_1} (u_1^{m_1-2}-v_1^{m_1-2})Q + R$, where $R\in\R^{(N-1)\times(N-1)}$ solves
\begin{align*}
  0 = \bar H(\bar u)
  \bigg(\frac{u_1^{2-m_1}}{n-1}(u_1^{m_1-2}-v_1^{m_1-2})Q 
  + R\bigg) + u_1^{m_1-2}QR
\end{align*}
or equivalently,
\begin{align*}
  R = -\frac{u_1^{2-m_1}}{n-1}(u_1^{m_1-2}-v_1^{m_1-2})
  (H(u)+u_1^{m_1-2}Q)^{-1} H(u)Q.
\end{align*}
It follows from the boundedness of the inverse and the condition $v_1\ge c>0$ that
\begin{align*}
  |W| &\le \bigg|\frac{u_1^{2-m_1}}{n-1}(u_1^{m_1-2}-v_1^{m_1-2})
  Q\bigg| + \bigg|\frac{u_1^{2-m_1}}{n-1}(u_1^{m_1-2}-v_1^{m_1-2})
  \bar H^*(\bar u)H(u)Q\bigg| \\
  &\le C(1+v_1^{m_1-2})|u_1^{2-m_1}-v_1^{2-m_1}|
  \le C(v_1)\sum_{i=2}^n|u_i-v_i|,
\end{align*}
where $C(v_1)$ depends on the infimum of $v_1$, which is positive (here, we use again $m_i\le 2$). This shows that
\begin{align*}
  I_{6} \le -\frac{\lambda}{2}\sum_{i=2}^N\int_\Omega
  |\na(u_i-v_i)|^2 dx
  + C(v_1)\sum_{i=2}^n\int_\Omega|u_i-v_i|^2|\na v|^2 dx.
\end{align*}

The terms $I_{7}$ and $I_{8}$ are estimated similarly as $I_2$ and $I_3$, respectively, in the proof of Theorem \ref{thm.wsu4}, leading to
\begin{align*}
  I_{7} &\le \delta\sum_{i=2}^N\int_\Omega|\na(u_i-v_i)|^2 dx
  + C(\delta)\sum_{i=2}^N\int_\Omega|u_i-v_i|^2|\na v|^2 dx, \\
  I_{8} &\le \delta C\sum_{i=2}^N\int_\Omega|u_i-v_i|^2
  (|\na v_1|^2 + |\na v_i|^2) dx
  + \delta C(v_1)\sum_{i=2}^N\int_\Omega|\na(u_i-v_i)|^2 dx.
\end{align*}
For the new terms $I_{9}$ and $I_{10}$, we use the identity $\sum_{k=2}^n(u_i-v_i)=-(u_1-v_1)$:
\begin{align*}
  I_{9} &+ I_{10} = \sum_{i=2}^n\sum_{j=2}^N\int_\Omega B_{ij}(u)
  (u_1^{m_1-2}\na u_1-v_1^{m_1-2}\na v_1)
  \cdot\na\frac{\pa\bar h_\eps}{\pa u_j}(\bar v)dx \\
  &\phantom{xx}+ \sum_{i=2}^n\sum_{j=2}^N\int_\Omega B_{ij}(v)
  \na\big(v_1^{m_1-2}(v_1-u_1)\big)\cdot\na
  \frac{\pa\bar h_\eps}{\pa u_j}(\bar v)dx \\
  &= \sum_{j=2}^N\int_\Omega\bigg(\na u_1 - \frac{u_1^{2-m_1}}{v_1^{2-m_1}}\na v_1\bigg)\sum_{i=2}^n
  \big(u_1^{m_1-2}B_{ij}(u)-v_1^{m_1-2}B_{ij}(v)\big)
  \cdot\na\frac{\pa\bar h_\eps}{\pa u_j}(\bar v)dx \\
  &\phantom{xx}- \sum_{j=2}^N\int_\Omega v_1^{2m_1-4}
  \big(u_1^{2-m_1} - (2-m_1)v_1^{1-m_1}(u_1-v_1) - v_1^{2-m_1}\big) \\
  &\phantom{xx}\times
  \sum_{j=2}^n B_{ij}(v)\na v_1\cdot
  \na \frac{\pa\bar h_\eps}{\pa u_j}(\bar v)dx.
\end{align*}
For the first term on the right-hand side, we use the Lipschitz continuity \eqref{3.map} of $u\mapsto u_1^{m_1-2}\sum_{i=2}^n B_{ij}(u)$ and Young's inequality, while we deduce from a Taylor estimate up to second order, the strict positivity of $v_1$, and the boundedness of $v$ for the second term that
\begin{align*}
  I_{9} &+ I_{10} \le
  \delta\bigg\|\na u_1 - \frac{u_1^{2-m_1}}{v_1^{2-m_1}}\na v_1
  \bigg\|_{L^2(\Omega)}^2 + C(\delta)\sum_{i,j=2}^N\int_\Omega
  |u_i-v_i|^2 \bigg|\na\frac{\pa\bar h_\eps}{\pa u_j}(\bar v)\bigg|^2 dx.
\end{align*} 
It follows from
\begin{align*}
  |a^{2-m_i}-b^{2-m_i}| &= (2-m_i)\bigg|\int_0^1(b+(a-b)\theta)^{1-m_i}
  (a-b)d\theta\bigg| \\
  &\le (2-m_i)\int_0^1((1-\theta)b)^{1-m_i}|a-b|d\theta 
  \le b^{1-m_i}|a-b|
\end{align*}
for $a\ge 0$ and $b>0$ that 
\begin{align*}
  \bigg|\na u_1 - \frac{u_1^{2-m_1}}{v_1^{2-m_1}}\na v_1\bigg|
  &\le |\na(u_1-v_1)| + v_1^{m_i-2}|u_1^{2-m_i}-v_1^{2-m_i}||\na v_1| \\
  &\le |\na(u_1-v_1)| + v_1^{-1}|u_1-v_1||\na v_1|,
\end{align*}
and consequently, because of the strict positivity of $v_1$ and of $u_1-v_1=-\sum_{i=2}^n(u_i-v_i)$,
\begin{align*}
  \bigg\|\na u_1 - \frac{u_1^{2-m_1}}{v_1^{2-m_1}}\na v_1
  \bigg\|_{L^2(\Omega)}^2
  &\le \|\na(u_1-v_1)\|_{L^2(\Omega)}^2 + C\|u_1-v_1\|_{L^2(\Omega)}^2
  \\
  &\le C\sum_{i=2}^n\|\na(u_i-v_i)\|_{L^2(\Omega)}^2
  + C\sum_{i=2}^n\|u_i-v_i\|_{L^2(\Omega)}^2.
\end{align*}
Furthermore, 
\begin{align*}
  \bigg|\na\frac{\pa\bar h_\eps}{\pa u_j}(v)\bigg|^2
  = \bigg|\sum_{k=2}^n(H_{jk}(v) + v_1^{m_1-2}Q_{jk})\na v_k\bigg|^2
  \le C\sum_{k=2}^n|\na v_k|^2 + C|\na v_1^{m_1-1}|^2.
\end{align*}
We conclude that
\begin{align*}
  I_9 + I_{10} &\le \delta C\sum_{i=2}^n\|\na(u_i-v_i)\|_{L^2(\Omega)}^2
  + \delta C\sum_{i=2}^n\|u_i-v_i\|_{L^2(\Omega)}^2 \\
  &\phantom{xx}+ C(\delta)\sum_{i=2}^n\int_\Omega|u_i-v_i|^2
  \bigg(\sum_{k=2}^n|\na v_k|^2 + |\na v_1^{m_1-1}|^2\bigg)dx.
\end{align*}
Estimating as in \eqref{2.GN}, using the Gagliardo--Nirenberg inequality, and collecting the previous inequalities, we find from \eqref{3.I1015} that  
\begin{align*}
  \frac{d}{dt}\int_\Omega \bar h_\eps(u|v)dx
  &\le -\bigg(\frac{\lambda}{2} - \delta C\bigg)
  \sum_{i=2}^N\int_\Omega|\na(u_i-v_i)|^2 dx \\
  &\phantom{xx}+ C\big(1+\|\na v\|_{L^p(\Omega)}^2
  + \|\na v_1^{m_1-1}\|_{L^p(\Omega)}^2\big)
  \sum_{i=2}^N\int_\Omega(u_i-v_i)^2 dx. 
\end{align*}
Choosing $\delta>0$ sufficiently small, the inequality $\bar h_\eps(u|v)\ge C\sum_{i=2}^N(u_i-v_i)^2$ yields
\begin{align*}
  \frac{d}{dt}\int_\Omega \bar h_\eps(u|v)dx
  \le C\big(1+\|\na v\|_{L^p(\Omega)}^2
  + \|\na v_1^{m_1-1}\|_{L^p(\Omega)}^2\big)\bar h_\eps(u|v).
\end{align*}
Our assumptions on $v$ allow us to apply the Gronwall inequality to deduce from $\bar h_\eps(u^0|v^0)=0$ that $\bar h_\eps(u(t)|v(t))=0$ and consequently $u(t)=v(t)$ in $\Omega$ for $t>0$.

It remains to consider the reaction rates. We need to estimate the following integrals associated to $I_4$ and $I_5$ in the proof of Theorem \ref{thm.wsu4}, 
\begin{align*}
  I_{11} &= \sum_{i=1}^n\int_\Omega(r_i(u)-r_i(v))
  \bigg(\frac{\pa\bar h_\eps}{\pa u_i}(\bar u)
  - \frac{\pa\bar h_\eps}{\pa u_i}(\bar v)\bigg)dx \\
  &= \sum_{i=2}^n\int_\Omega(r_i(u)-r_i(v))
  \bigg(\frac{\pa\bar h_\eps}{\pa u_i}(\bar u)
  - \frac{\pa\bar h_\eps}{\pa u_i}(\bar v)\bigg)dx \\
  &\phantom{xx}+ \int_\Omega(r_1(u)-r_1(v))
  (h_1'(u_1)-h_1'(v_1))dx =: I_{111} + I_{112}, \\
  I_{12} &= \sum_{i=1}^n\int_\Omega r_i(v)\bigg(
  \frac{\pa\bar h_\eps}{\pa u_i}(\bar u)
  - \frac{\pa\bar h_\eps}{\pa u_i}(\bar v)
  - \sum_{k=2}^n\frac{\pa^2\bar h_\eps}{\pa u_i\pa u_k}(\bar v)
  (u_k-v_k)\bigg)dx \\
  &= \sum_{i=2}^n\int_\Omega r_i(v)\bigg(
  \frac{\pa\bar h_\eps}{\pa u_i}(\bar u)
  - \frac{\pa\bar h_\eps}{\pa u_i}(\bar v)
  - \sum_{k=2}^n\frac{\pa^2´\bar h_\eps}{\pa u_i\pa u_k}(\bar v)
  (u_k-v_k)\bigg)dx \\
  &\phantom{xx}+ \int_\Omega r_1(v)
  \bigg(h_1'(u_1) - h_1'(v_1) + (m_1-1)v_1^{m_1-2}
  \sum_{k=2}^n(u_k-v_k)\bigg)dx \\
  &=: I_{121} + I_{122}.
\end{align*}
The integrals $I_{111}$ and $I_{121}$ are estimated as in Theorem \ref{thm.wsu4}. We use Assumption (A6)$_3$ for the term $I_{112}$: 
\begin{align*}
  I_{112} \le C\sum_{i=0}^N\int_\Omega(u_i-v_i)^2dx.
\end{align*}
Finally, by a Taylor expansion,
\begin{align*}
  I_{122} &= -\int_\Omega r_1(v)
  \big(h_1'(u_1) - h_1'(v_1) - (m_1-1)v_1^{m_1-2}(u_1-v_1)\big)dx \\
  &\le C\int_\Omega |u_1-v_1|^2 dx 
  \le C\sum_{i=2}^n\int_\Omega(u_i-v_i)^2 dx. 
\end{align*}
From this point on, we can argue as before, finishing the proof.


\section{Examples}\label{sec.ex} 

We present some examples to which Theorems \ref{thm.wsu1}--\ref{thm.wsu3} can be applied. Sections \ref{sec.BT}--\ref{sec.chemo} concern systems without volume filling, while examples with volume-filling effects are given in Sections \ref{sec.granu}--\ref{sec.chemovf}. 

\subsection{Regularized Busenberg--Travis model}\label{sec.BT}

The regularized Busenberg--Travis equations for the segregating population densities $u_1,\ldots,u_n$ (we set $N=n$) are given by \eqref{1.eq} with the diffusion coefficients 
\begin{align*}
  A_{ij}(u) = a_{i0}\delta_{ij} + a_{ij}u_i
  \quad\mbox{for }i,j=1,\ldots,n,
\end{align*}
where $a_{i0}>0$, $a_{ij}\in\R$ are numbers. The original model in \cite{BuTr83} fulfills $a_{ij}=1$ and does not contain the term $a_{i0}\delta_{ij}$. We consider a more general, regularized model. In particular, we suppose that $a_{i0}>0$ and that $(a_{ij})$ is positive semidefinite. Since $(h''(u)A(u))_{ij} = a_{i0}u_i^{-1}\delta_{ij} + a_{ij}$, this implies the positive definiteness of $h''(u)A(u)$, verifying Assumption (A5)$_1$. Thus, Theorem \ref{thm.wsu1} yields the weak--strong uniqueness property. 

The global existence of a weak solution to this system with initial and boundary conditions \eqref{1.bic} was proved in \cite[Appendix B]{JPZ22}. Compared to \cite{LaMa23}, we are able to relax the conditions imposed on the strong solution $v$ and to extend the result to an arbitrary number $n\geq 2$ of species. The boundedness of a weak solution to the two-species model was proved in \cite{LaMa23a}. It is an open problem whether the model for more than two species generally admits a bounded weak solution.

\subsection{Shigesada--Kawasaki--Teramoto (SKT) model}

The SKT equations also describe segregating population species \cite{SKT79}. They are given by \eqref{1.eq} with the diffusion coefficients
\begin{align*}
  A_{ij}(u) = \delta_{ij}\bigg(a_{i0} + \sum_{k=1}^n a_{ik}u_k\bigg)
  + a_{ij}u_i \quad\mbox{for } i,j=1,\ldots,n,
\end{align*}
with $N=n$, $a_{ij}\ge 0$, $a_{i0}>0$, and $(a_{ij})$ is positive definite. The computation
\begin{align*}
  z^Th''(u)A(u)z = \sum_{i=1}^n
  \bigg(a_{i0}+\sum_{k=1}^n a_{ik}u_k\bigg)\frac{z_i^2}{u_i}
  + \sum_{i,j=1}^n a_{ij}z_iz_j
\end{align*}
for $z\in\R^n$ shows that $h''(u)A(u)$ is positive definite; also see \cite[Lemma 4]{CDJ18}. It is shown in \cite{CDJ18} that there exists a global weak solution. If this solution is bounded, Theorem \ref{thm.wsu1} applies, and we obtain the weak--strong uniqueness property. The boundedness of weak solutions can be proved for certain choices of the coefficients $a_{ij}$ in the two-species SKT model \cite{JuZa16}. 

If $n=2$ and the $L^\infty(\Omega)$ norm of $u$ is smaller than a certain constant depending on the diffusion coefficients, the uniqueness of weak solutions was proved in \cite[Theorem 2.1]{BMM25}. Compared to \cite{ChJu19}, we are able to relax the condition imposed on the strong solution. The local-in-time existence of smooth solutions to the two-species model (in particular $v_i\in C^0([0,T_{\rm max});W^{1,p}(\Omega))$ with $p>d$) was shown in \cite{HNP15}.

\subsection{Semiconductor model with electron--hole scattering}\label{sec.semicond}

The transport of carriers (electrons and holes) in semiconductors with electron--hole scattering can be modeled by drift--diffusion equations for the electron and hole densities $u_1$ and $u_2$, respectively. To simplify, we focus on the diffusion part and neglect the drift part due to the electric field. Then the model is given by equations \eqref{1.eq} with the diffusion matrix $A(u)=(A_{ij})$ with
\begin{align*}
  A_{11} = \frac{\mu_1(1+\mu_2 u_1)}{m(u)}, \quad
  A_{12} = \beta\frac{\mu_1\mu_2u_1}{m(u)}, \quad
  A_{21} = \beta\frac{\mu_1\mu_2u_2}{m(u)}, \quad
  A_{22} = \frac{\mu_2(1+\mu_1 u_2)}{m(u)}.
\end{align*}
where $\mu_1$, $\mu_2>0$, $\beta\in[0,1]$, and 
\begin{align*}
  m(u) = 1 + \mu_2u_1 + \mu_2u_1 + (1-\beta^2)\mu_1\mu_2u_1u_2.
\end{align*}
A global existence analysis for the case $\beta=1$ was presented in \cite{ChJu07}. We compute (with $N=n=2$)
\begin{align*}
  z^Th''(u)A(u)z 
  &= \frac{\mu_1}{m(u)}\bigg(\frac{1}{u_1}+\mu_2\bigg)z_1^2
  + \frac{\mu_2}{m(u)}\bigg(\frac{1}{u_2}+\mu_1\bigg)z_2^2
  + 2\beta\frac{\mu_1\mu_2}{m(u)}z_1z_2 \\
  &\ge \sum_{i=1}^2\frac{\mu_i}{m(u)}z_i^2
  + \beta\frac{\mu_1\mu_2}{m(u)}(z_1+z_2)^2 
  \ge c\sum_{i=1}^2 z_i^2,
\end{align*}
for all $z\in\R^2$ and $u\in(0,L)^2$ for some $L>0$. Thus, Assumption (A5)$_1$ is satisfied, and Theorem \ref{thm.wsu1} applies. The proof of the boundedness of the weak solution is an open problem. 

\subsection{Chemotaxis model with additional cross-diffusion}
\label{sec.chemo}

The Keller--Segel equations model the chemotactic dynamics of the cell density $u_1$ of bacteria driven by the gradient of the chemical concentration $u_2$. The nonlocal chemical interaction may dominate diffusion resulting in the blow-up of the cell density. The blow-up can be avoided by adding an additional cross-diffusion term with parameter $\delta>0$, leading to a global existence result \cite{HiJu11}. This gives the equations
\begin{align*}
  \pa_t u_1 = \diver(\na u_1-u_1\na u_2), \quad
  \pa_t u_2 = \Delta u_1 + \delta\Delta u_2 = u_1-u_2.
\end{align*}
The diffusion matrix equals
\begin{align*}
  A(u) = \begin{pmatrix} 1 & -u_1 \\ \delta & 1 \end{pmatrix}.
\end{align*}
The model corresponds to equations \eqref{1.eq} with $n=1$ and $N=2$, and the entropy density is given by
\begin{align*}
  h(u) = u_1(\log u_1-1) + \frac{u_2^2}{2\delta}
  \quad\mbox{for }u_1,u_2\ge 0.
\end{align*}
Under the condition that $u_1$ is bounded, we find that
\begin{align*}
  z^Th''(u)A(u)z = \frac{z_1^2}{u_1} + \frac{z_2^2}{\delta}
  \ge c_A|z|^2 \quad\mbox{for }z\in\R^2,
\end{align*}
where $c_A=\min\{1/\max|u_2|,1/\delta\}>0$. Thus, Assumption (A5)$_1$ holds, and Theorem \ref{thm.wsu1} yields the weak--strong uniqueness property. We infer from Amann's result \cite{Ama93}, for sufficiently smooth positive initial data, that there exists a local-in-time strong solution.

\subsection{Granular segregation model}\label{sec.granu}

Granular materials of different components may segregate under external agitation. Denoting by $u_1$ the difference of two densities and by $u_2$ the dynamic angle between the flat free surface of the mixture and the horizontal axis, the dynamics is governed by equations \eqref{1.eq} with the diffusion matrix  
\begin{align*}
  \bar A(\bar u) = \begin{pmatrix} 
  \nu & -(1-u_1^2) \\ \gamma & 1 \end{pmatrix},
\end{align*}
where $\nu>0$ and $\gamma>0$ are related to Fickian diffusivities. We refer to \cite{GJV03} for the proof of the global existence of a weak solution. In this example, we choose the entropy density
\begin{align*}
  \bar h(\bar u) = \frac12(1-u_1)(\log(1-u_1)-1)
  + \frac12(1+u_1)(\log(1+u_1)-1) + \frac{u_2^2}{2\gamma},
\end{align*}
where $u_1\in(-1,1)$ and $u_2\in\R$. The model is already formulated in the reduced variable $\bar u = (u_1,u_2)$, and the original variables are $v_1=\frac12(1-u_1)$, $v_2=1-v_1=\frac12(1+u_1)$, and $v_3=u_2$. Therefore, we have $n=2$ and $N=3$. The matrix
\begin{align*}
  \bar h''(\bar u)\bar A(\bar u)
  = \begin{pmatrix}
  \nu(1-u_1^2)^{-1} & -1 \\ 1 & 1/\gamma \end{pmatrix}
\end{align*}
is positive definite, $\bar z^T\bar h''(\bar u)\bar A(\bar u)\bar z \ge \min\{\nu,1/\gamma\}|\bar z|^2$ for $\bar z\in\R^2$. According to Lemma \ref{lem.bar}, also the matrix $h''(v)A(v)$ is positive definite uniformly in $\mathcal{D}= \{v\in(0,1)^2\times(0,L):v_1+v_2\le 1\}$. Therefore, we deduce from Theorem \ref{thm.wsu2} the weak--strong uniqueness property. 

\subsection{Maxwell--Stefan model for gas mixtures}

The dynamics of gaseous mixtures can be modeled by the equations 
\begin{align*}
  \pa_t u_i + \diver J_i = 0, \quad
  \na u_i = \sum_{j=1}^n d_{ij}(u_jJ_i-u_iJ_j),
\end{align*}
where $u_i$ are the volume fractions of the gas components, $J_i$ are the associated fluxes, and the diffusion coefficients $d_{ij}$ are symmetric. The existence of a bounded weak solution satisfying $\sum_{i=1}^n u_i=1$ was established in \cite{JuSt13}. The system can be formulated in the form \eqref{1.eq} with $n=N$ and the diffusion coefficients $A_{ij}(u)=\sqrt{u_i/u_j}D_{ij}^{BD}(u)$, where $D^{BD}(u)$ is the Bott--Duffin inverse of $D(u)$ with the coefficients
\begin{align*}
  D_{ij}(u) = \delta_{ij}\sum_{k\neq i}d_{ik}u_k - d_{ij}\sqrt{u_iu_j}
  \quad\mbox{for }i,j=1,\ldots,n;
\end{align*}
see, e.g., \cite[Sec.~2]{HJT22}. It is shown in \cite[Lemma 4]{HJT22} that $D^{BD}(u)$ is hypocoercive in the sense that there exists $c>0$ such that for all $z\in\R^n$ and $u\in(0,1)^n$ with $\sum_{i=1}^n u_i=1$,
\begin{align*}
  z^TD^{BD}(u)z \ge c|\Pi z|^2,
\end{align*}
where $\Pi$ is the projection on $\{z\in\R^n: \sum_{i=1}^n \sqrt{u_i}z_i=0\}$, given by $\Pi_{ij}=\delta_{ij}-\sqrt{u_iu_j}$. It follows that
\begin{align*}
  z^Th''(u)A(u)z &= \sum_{i,j=1}^n D_{ij}^{BD}(u)\frac{z_i}{\sqrt{u_i}}
  \frac{z_j}{\sqrt{u_j}}
  \ge c\sum_{i=1}^n\bigg(\frac{z_i}{\sqrt{u_i}}
  - \sum_{j=1}^n\sqrt{u_i}z_j\bigg)^2 \\
  &= c\sum_{i=1}^n \frac{z_i^2}{u_i} - c\bigg(\sum_{i=1}^n u_i\bigg)
  \bigg(\sum_{j=1}^nz_j\bigg)^2
  \ge c\sum_{i=1}^n z_i^2 - c\bigg(\sum_{j=1}^nz_j\bigg)^2,
\end{align*}
proving Assumption (A5)$_2$. Hence, Theorem \ref{thm.wsu2} applies, and the weak--strong uniqueness property follows.

\subsection{Solar-cell thin-film model} 

Thin films for solar cells can be prepared by physical vapor deposition on a flat substrate. The diffusion of the metallic species in the bulk can be described by equations \eqref{1.eq} with the diffusion coefficients
\begin{align*}
  A_{ij}(u) = \delta_{ij}\sum_{k=1}^n a_{ik}u_k - a_{ij}u_i \quad\mbox{for }i,j=1,\ldots,n,
\end{align*}
where $a_{ij}=a_{ji}>0$. Set $a_*=\min_{i,j=1,\ldots,n}a_{ij}$. The model satisfies $N=n$ and the volume filling-constraint $\sum_{i=1}^n u_i=1$. The existence of a bounded weak solutions was shown in \cite{BaEh18}. The computation
\begin{align*}
  z^Th''(u)A(u)z &= \sum_{i=1}^n\bigg(\sum_{j=1}^n a_{ij}u_j\bigg)
  \frac{z_i^2}{u_i} - \sum_{i,j=1}^n a_{ij}z_iz_j
  = \frac12\sum_{i,j=1}^n a_{ij}u_iu_j\bigg(\frac{z_i}{u_i}
  - \frac{z_j}{u_j}\bigg)^2 \\
  &\ge \frac{a_*}{2}\sum_{i,j=1}^n u_iu_j\bigg(\frac{z_i}{u_i}
  - \frac{z_j}{u_j}\bigg)^2 
  = a_*\bigg\{\sum_{i=1}^n\frac{z_i^2}{u_i} 
  - \sum_{i=1}^nz_i\sum_{j=1}^n z_j\bigg\} \\
  &\ge a_*\bigg\{\sum_{i=1}^n z_i^2 
  - \bigg(\sum_{i=1}^n z_i\bigg)^2\bigg\}
\end{align*}
for $z\in\R^n$ and $u\in(0,1)^n$ with $\sum_{i=1}^n u_i=1$. This proves Assumption (A5)$_2$. By Theorem \ref{thm.wsu2}, the weak--strong uniqueness property holds. 

\subsection{Multiphase model}

The interaction between different components of a cellular fluid, like various cell types, extracellular matrix, enzymes, and water, can be described by equations \eqref{1.eq} for the volume fraction $u_i$ of the $i$th phase with the diffusion coefficients
\begin{align*}
  A_{ij}(u) = \delta_{ij}\sum_{k=1}^n a_{ik}u_k + a_{ij}u_i
  - 2u_i\sum_{k=1}^n a_{jk}u_k
  \quad\mbox{for }i,j=1,\ldots,n,
\end{align*}
where $a_{ij}\ge 0$ and $(a_{ij})$ is positive definite with smallest eigenvalue $\alpha>0$ \cite{JRX26}. We have $n=N$ and $\sum_{i=1}^n u_i=1$. The existence of a bounded weak solution was proved in \cite{JRX26}. Setting $a_i(u)=\sum_{j=1}^n a_{ij}u_j$, we obtain for $z\in\R$ and $u\in(0,1)^n$ with $\sum_{i=1}^n u_i=1$:
\begin{align*}
  z^Th''(u)A(u)z &= \sum_{i=1}^n\frac{a_i(u)}{u_i}z_i^2
  + \sum_{i,j=1}^n a_{ij}z_iz_j 
  - 2\sum_{i,j=1}^n a_j(u)z_iz_j \\
  &\ge \sum_{i=1}^n \frac{a_i(u)}{u_i}z_i^2 
  + \alpha\sum_{i=1}^n z_i^2
  - 2\bigg(\sum_{i=1}^n z_i\bigg)\bigg(\sum_{j=1}^n a_j(u)z_j\bigg) \\
  &\ge \alpha\sum_{i=1}^n z_i^2
  - \eta\bigg(\sum_{j=1}^n a_j(u)z_j\bigg)^2
  - \frac{1}{\eta}\bigg(\sum_{i=1}^n z_i\bigg)^2 \\
  &\ge \sum_{i=1}^n\Big(\alpha
  - \eta C\max_{k=1,\ldots,n}a_{ik}^2\Big)z_i^2
  - \frac{1}{\eta}\bigg(\sum_{i=1}^n z_i\bigg)^2 \\
  &\ge \lambda\sum_{i=1}^n z_i^2 - \frac{1}{\eta}\bigg(\sum_{i=1}^n z_i\bigg)^2,
\end{align*}
where we used Young's inequality with parameter $\eta>0$ and we have chosen $\eta>0$ sufficiently small such that $\lambda:=\alpha/2\le \alpha - \eta C\max_{i,k=1,\ldots,n}a_{ik}^2$. Thus, we have verified Assumption (A5)$_2$ with $\kappa:=1/\eta$. The weak--strong uniqueness property follows from Theorem \ref{thm.wsu2}.

\subsection{Multispecies chemotaxis model with volume filling}
\label{sec.chemovf}

We consider a chemotaxis-driven multiphase multispecies diffusion system, which can be derived within a multiphase framework based on mass and force balance laws, similarly to \cite{GJZ26}. The equations for the volume fractions $u_1,\ldots,u_n$ of the cellular components are given by
\begin{align}\label{4.chemo1}
  & \pa_t u_i = \diver\bigg(J_i - u_i\sum_{k=1}^n J_k\bigg), \\
  & J_i = f_i(u_i)\na u_i + \na(u_iq_i(c))
  + u_i\sum_{k=1}^g \chi_{ik}\na c_k, 
  \quad i=1,\ldots,n, \label{4.chemo2}
\end{align}
where the concentrations $c_j$ of the chemical signals are determined by
\begin{align}\label{4.chemo3}
  \pa_t c_j = D_j\Delta c_j - \lambda_jc_j 
  + \sum_{k=1}^n \mu_{jk}u_k, \quad j=1,\ldots,g.
\end{align}
The parameters $\chi_{ik}\in\R$ model the sensitivity of the cellular components modeling direct aggregation or repulsion of the $i$th cell species, $\mu_{jk}\ge 0$ is the cell production rate, $D_j>0$ is the chemical diffusion coefficient, and $\lambda_j>0$ is the signal degradation rate. The model is similar to that one in \cite{GJZ26}; the main difference is that our model is non-degenerate. The volume fractions satisfy the constraint $\sum_{i=1}^n u_i=1$. The entropy density is given by \eqref{1.h} with $b_{ij}=K>0$. The existence of a bounded weak solution and the weak--strong uniqueness property can be shown similarly as in \cite{GJZ26}. Here, we improve the latter result by allowing for a weaker notion of strong solution. 

We wish to apply Theorem \ref{thm.wsu3}. For this, we set $N=n+g$ and $v=(u,c)$. The diffusion matrix reads as
\begin{align*}
  A(u) = \begin{pmatrix}
  A^{11}(v) & A^{12}(v) \\ 0 & A^{22}(v)
  \end{pmatrix},
\end{align*}
where 
\begin{align*}
  A^{11}_{ik}(v) &= (q_i(c) + f_i(u_i))\delta_{ik}
  - (q_k(c)+f_k(u_k))u_k, \\
  A^{12}_{ij}(v) &= u_i\frac{\pa q_i}{\pa c_j}(c) + \chi_{ij}u_i
  - u_i\sum_{k=1}^n u_k\bigg(\frac{\pa q_k}{\pa c_j}(c)
  + \chi_{kj}\bigg), \\
  A^{22}_{j\ell}(v) &= D_\ell\delta_{j\ell}
  \quad\mbox{for }i,k=1,\ldots,n,\ j,\ell=1,\ldots,g.
\end{align*}
Thus, we can formulate system \eqref{4.chemo1}--\eqref{4.chemo3} as \eqref{1.eq} with the reaction rates
\begin{align*}
  r_i(v) &= 0\quad\mbox{for }i=1,\ldots,n, \\
  r_{n+j}(v) &= -\lambda_j v_{n+j} + \sum_{k=1}^n\mu_{jk}v_{n+k}
  \quad\mbox{for }j=1,\ldots,g
\end{align*}
and the volume-filling constraint $\sum_{i=1}^n v_i=1$. 

To show the requirements of Theorem \ref{thm.wsu3}, we assume that 
\begin{align*}
  q_i(c)+f_i(u_i)\ge c>0 \quad\mbox{for }i=2,\ldots,n, \quad
  q_1(c)+f_1(u_1)\ge 0,
\end{align*}
$f_i$ is continuous, $q_i$ is continuously differentiable, and $f_i$, $q_i$, and $q_i'$ are bounded for $i=1,\ldots,n$. The variable $u_1$ may represent the volume fraction of water, which may not undergo diffusion. Consequently, it is possible to have $f_1=q_1=0$, leading to a lack of coercivity in this variable. 

\begin{lemma}
If $K>0$ is sufficiently large, Assumptions (A3), (A4), (A5)$_3$, and (A6) hold.
\end{lemma}

\begin{proof}
Assumption (A3) is an immediate consequence of the definition of $A^{k\ell}_{ij}(v)$ with $a_i(v)=q_i(c)+f_i(u_i)\ge c>0$ and $a_{ij}(v)$ contains the factor $v_i$ for $i=1,\ldots,n$ such that $m_i=1$. Moreover, the linearity of $r_j$ for $j=n+1,\ldots,N$ ensures the validity of Assumption (A4), and Assumption (A6) follows directly from $r_1=0$.

To verify Assumption (A5)$_3$, let $z\in\R^N$. We infer from the bounds for $f_i$ and $q_i$ and Young's inequality with $\eta_1>0$ that
\begin{align*}
  z^Th''(u)A(u)z &= \sum_{i=1}^n(q_i(c)+f_i(v_i))\frac{z_i^2}{v_i}
  + \sum_{i=1}^n\sum_{j=n+1}^N\bigg(\frac{\pa q_i}{\pa v_j}(c)
  + \chi_{i,j-n}\bigg)z_iz_j \\
  &\phantom{xx}- \bigg(\sum_{i=1}^n z_i\bigg)
  \sum_{i=1}^n\bigg\{(q_i(c)+f_i(v_i))z_i
  + \sum_{j=n+1}^N v_i\bigg(\frac{\pa q_i}{\pa v_j}(c) 
  + \chi_{i,j-n}\bigg)z_j\bigg\} \\
  &\phantom{xx}+ K\sum_{j=n+1}^N D_{j-n}z_j^2 \\
  &\ge (c - \eta_1)\sum_{i=2}^n z_i^2 
  - C(\eta_1)\bigg(\sum_{i=1}^n z_i\bigg)^2
  - C\sum_{i=1}^n\sum_{j=n+1}^N|z_iz_j| + KD_*\sum_{j=n+1}^Nz_j^2,
\end{align*}
where $D_*=\min_{j=1,\ldots,g}D_j>0$. We estimate the mixed term:
\begin{align*}
  \sum_{i=1}^n\sum_{j=n+1}^N|z_iz_j|
  &= \sum_{i=2}^n\sum_{j=n+1}^N|z_iz_j| 
  + |z_1|\sum_{j=n+1}^n|z_j| \\
  &= \sum_{i=2}^n\sum_{j=n+1}^N|z_iz_j|
  + \bigg|\sum_{i=1}^nz_i - \sum_{i=2}^n z_i\bigg|\sum_{j=n+1}^n|z_j| \\
  &\le \eta_2\sum_{i=2}^nz_i^2 + C(\eta_2)\sum_{j=n+1}^N z_j^2
  + \bigg|\sum_{i=1}^n z_i\bigg|\sum_{j=n+1}^n|z_j|
  + \bigg|\sum_{i=2}^n z_i\bigg|\sum_{j=n+1}^n|z_j| \\
  &\le 2\eta_2\sum_{i=2}^nz_i^2
  + C(\eta_2)\sum_{j=n+1}^N z_j^2 
  + C(\eta_2)\bigg|\sum_{i=1}^n z_i\bigg|^2,
\end{align*}
where $\eta_2>0$. This shows that
\begin{align*}
  z^Th''(u)A(u)z &\ge (c-\eta_1-2\eta_2)\sum_{i=2}^nz_i^2
  + (KD_*-C(\eta_2))\sum_{j=n+1}^N z_j^2 \\
  &\phantom{xx}- (C(\eta_1)+C(\eta_2))\bigg(\sum_{i=1}^n z_i\bigg)^2.
\end{align*}
Thus, choosing $\eta_1>0$ and $\eta_2>0$ sufficiently small and $K>0$ sufficiently large, we find that
\begin{align*}
  z^Th''(u)A(u)z \ge \frac{c}{2}\sum_{i=2}^N z_i^2 
  - (C(\eta_1)+C(\eta_2))\bigg(\sum_{i=1}^n z_i\bigg)^2.
\end{align*}
This shows Assumption (A5)$_3$ with $\lambda := c/2$ and $\kappa := C(\eta_1)+C(\eta_2)$, completing the proof.
\end{proof}


\end{document}